\documentclass[a4paper,12pt]{article}
\usepackage[cp1251]{inputenc}
\usepackage{amsfonts, amssymb, amsmath, amsthm, amscd}
\usepackage{cite}
\ifx\undefined \pdfgentounicode \else
\input{glyphtounicode} \pdfgentounicode=1
\fi

\author{A.A. Vasil'eva}
\title{Kolmogorov widths of weighted Sobolev classes on a
multi-dimensional domain with conditions on the derivatives of
order $r$ and zero}
\date{}
\begin{document}

\maketitle

\newenvironment{Biblio}{%
                  \renewcommand{\refname}{\footnotesize REFERENCES}%
                  \begin{thebibliography}{99}%
                  \itemsep -0.2em}%
                  {\end{thebibliography}}

\def\inff{\mathop{\smash\inf\vphantom\sup}}
\renewcommand{\le}{\leqslant}
\renewcommand{\ge}{\geqslant}
\newcommand{\sgn}{\mathrm {sgn}\,}
\newcommand{\inter}{\mathrm {int}\,}
\newcommand{\dist}{\mathrm {dist}}
\newcommand{\supp}{\mathrm {supp}\,}
\newcommand{\R}{\mathbb{R}}
\newcommand{\Z}{\mathbb{Z}}
\newcommand{\N}{\mathbb{N}}
\newcommand{\Q}{\mathbb{Q}}
\theoremstyle{plain}
\newtheorem{Trm}{Theorem}
\newtheorem{trma}{Theorem}
\newtheorem{Def}{Definition}
\newtheorem{Cor}{Corollary}
\newtheorem{Lem}{Lemma}
\newtheorem{Rem}{Remark}
\newtheorem{Sta}{Proposition}
\newtheorem{Sup}{Assumption}
\newtheorem{Supp}{Assumption}
\newtheorem{Exa}{Example}
\renewcommand{\proofname}{\bf Proof}
\renewcommand{\thetrma}{\Alph{trma}}
\renewcommand{\theSupp}{\Alph{Supp}}

\section{Introduction}

The problem on estimating the Kolmogorov and linear widths of
weighted Sobolev classes is studied since 1970's \cite{el_kolli,
trieb_mat_sb}. These classes can be defined differently, depending
on smaller-order derivatives and boundary conditions. For example,
in \cite{triebel12, mieth1, mieth2, vas_width_raspr} the weighted
Sobolev classes are defined only by conditions on the higher-order
derivatives; in \cite{lo1, lo2, lo3, triebel, trieb_mat_sb, boy_1,
myn_otel, ait_kus1, ait_kus2} they are defined by conditions on
the derivatives of different orders. Also notice that in
\cite{step_lom, edm_lang, konov_leviat, lang_j_at1, lif_linde} the
weighted Sobolev spaces on an interval or a semi-axis are defined
as the image of a weighted Riemann-Liouville operator (the
criterion of boundedness of such operators was obtained by V.D.
Stepanov \cite{stepanov2, stepanov1}).

Here we consider the weighted Sobolev spaces with conditions on
the highest order and zero derivatives.

First we recall the definitions of the class $W^r_{p,g}(\Omega)$
and the space $L_{q,v}(\Omega)$.

Let $\Omega \subset \R^d$ be a domain, and let $g$,
$v:\Omega\rightarrow (0, \, \infty)$ be measurable functions.
Given a distribution $f$ on $\Omega$, we set $\displaystyle \nabla
^r\!f=\left(\partial^{r}\! f/\partial
x^{\overline{\beta}}\right)_{|\overline{\beta}| =r}$ (the partial
derivatives are taken in a sense of distributions;
$\overline{\beta} =(\beta_1, \, \dots, \, \beta_d)$,
$|\overline{\beta}| =\beta _1+ \ldots+\beta _d$). Let $l_{r,d}$ be
the number of components of the vector-valued distribution $\nabla
^r\!f$. We set
$$
W^r_{p,g}(\Omega)=\left\{f:\ \Omega\rightarrow \R\big| \; \exists
\psi :\ \Omega\rightarrow \R^{l_{r,d}}\!:\ \| \psi \|
_{L_p(\Omega)}\le 1, \, \nabla ^r\! f=g\cdot \psi\right\}
$$
\Big(we denote the corresponding function $\psi$ by
$\displaystyle\frac{\nabla ^r\!f}{g}$\Big),
$${\cal W}^r_{p,g}(\Omega)={\rm span}\,W^r_{p,g}(\Omega),$$
$$
\| f\|_{L_{q,v}(\Omega)}{=}\| fv\|_{L_q(\Omega)},\qquad
L_{q,v}(\Omega)=\left\{f:\Omega \rightarrow \R| \; \ \| f\|
_{L_{q,v}(\Omega)}<\infty\right\}.
$$

We define the set $M$ as the intersection of the class
$W^r_{p_1,g}(\Omega)$ and the unit ball of the space
$L_{p_0,w}(\Omega)$; i.e.,
\begin{align}
\label{m_def} M = \left\{ f:\Omega \rightarrow \R:\; \left\|
\frac{\nabla ^r f}{g}\right\|_{L_{p_1}(\Omega)}\le 1, \quad
\|wf\|_{L_{p_0}(\Omega)}\le 1\right\}.
\end{align}

For $d=1$, $p_1=p_0>1$, $q\ge 1$, the criterion for boundedness of
$M$ in the space $L_{q,v}(\Omega)$ was obtained by R.O. Oinarov
\cite{r_oinarov}. Then this result was generalized by V.D.
Stepanov and E.P. Ushakova \cite{st_ush} for $0<p_0\le q$,
$1<p_1\le q<\infty$ and $0<q<p_1<\infty$, $p_0=p_1>1$.

For multi-dimensional domain problems on embeddings of weighted
Sobolev classes with restrictions on derivatives of different
orders were studied by A. Kufner, P.I. Lizorkin, M.O. Otelbaev, K.
Mynbaev, L.K. Kusainova, O.V. Besov and other authors (see, e.g.,
\cite{lo1, myn_otel, besov, caso_ambr, kus1, kufner}).

We recall that the Kolmogorov $n$-width of a subset $C$ in a
normed space $X$ is the quantity
$$
d_n(C, \, X) = \inf _{L\in {\cal L}_n(X)} \sup _{x\in C} \inf
_{y\in L}\|x-y\|;
$$
here ${\cal L}_n(X)$ is a family of subspaces in $X$ of dimension
at most $n$; the linear $n$-width is the quantity
$$
\lambda_n(C, \, X) = \inf _{A\in L(X, \, X), \, {\rm rk}\, A\le n}
\sup _{x\in C}\|x-Ax\|
$$
(here $L(X, \, X)$ is the family of linear continuous operators on
$X$, ${\rm rk}\, A$ is the dimension of the image of $A$).

Triebel \cite{triebel, trieb_mat_sb} obtained the estimates for
the Kolmogorov widths of the set $M$ for $p_1=p_0\le q$; $\Omega$
is a domain with smooth boundary, the weights are the powers of
the distance from $\partial \Omega$. The parameters are such that
the orders of widths depend only on the conditions on the
highest-order derivatives. In the papers of Lizorkin and Otelbaev
\cite{lo3}, Aitenova and Kusainova \cite{ait_kus1, ait_kus2}, in
the book of Mynbaev and Otelbaev \cite{myn_otel} estimates for the
linear widths of the set $M$ in $L_{q,v}$ were obtained for
$p_0=p_1$ and general weights. For $q\le 2$ or $p_1\ge 2$ under
some conditions on the weights the upper and the lower estimates
are the same in the sense of orders. In addition, in
\cite{myn_otel} the special case $\Omega = \R^d$,
$g(x)=(1+|x|)^\beta$, $w(x)=(1+|x|)^\sigma$,
$v(x)=(1+|x|)^\lambda$ was considered (again for $p_0=p_1$; for
$p_1<2<q$ the upper and the lower estimates are different in the
sense of orders). Boykov \cite{boy_1} obtained the order estimates
for the Kolmogorov widths of the classes $\cap _{k=0}^r
W^k_{p_k,g_k}(K)$, where $p_k=\infty$ for $0\le k\le l$, $p_k=p$
for $l+1\le k\le r$, $K$ is a cube, $g_i$ are the powers of the
distance from $\partial K$. The conditions on the parameters are
such that the orders of the $n$-widths depend only on the
restrictions on the high-order derivatives.

In this paper we obtain the order estimates for the Kolmogorov
widths of the set $M$ in the space $L_{q,v}(\Omega)$. In the first
two examples $\Omega$ is a John domain, the weights are the
functions of distance from some $h$-subset of $\partial \Omega$
(the necessary definitions will be given later). In the third
example $\Omega=\R^d$, the weights are powers of $1+|x|$ (as in
\cite{myn_otel}).

We denote by $B_a(x)$ the euclidean ball of radius $a$ with the
center in the point $x$.

\begin{Def}
\label{fca} Let $\Omega\subset\R^d$ be a bounded domain, $a>0$. We
write $\Omega \in {\bf FC}(a)$ if there is a point $x_*\in \Omega$
such that for all $x\in \Omega$ there are a number $T(x)>0$ and a
curve $\gamma _x:[0, \, T(x)] \rightarrow\Omega$ with the
following properties:
\begin{enumerate}
\item $\gamma _x$ has the natural parametrization with respect to
the euclidean norm on $\R^d$,
\item $\gamma _x(0)=x$, $\gamma _x(T(x))=x_*$,
\item $B_{at}(\gamma _x(t))\subset \Omega$ for all $t\in [0, \, T(x)]$.
\end{enumerate}
We say that $\Omega$ is a John domain (or satisfies the John
condition) if $\Omega\in {\bf FC}(a)$ for some $a>0$.
\end{Def}

As examples of such domains we can take bounded domains with
Lipschitz boundary and the Koch's snowflake. The domain $\{(y, \,
z)\in \R^{d-1}\times \R:\; 0<z<1, \; |y|<z^\sigma\}$ for
$\sigma>1$ does not satisfy the John condition. Yu.G. Reshetnyak
\cite{resh1, resh2} proved that for a John domain the condition
for embedding of a non-weighted Sobolev class into a non-weighted
Lebesgue space is the same as for a cube.

\begin{Def}
\label{h_set} {\rm (see \cite{m_bricchi1}).} Let $\Gamma\subset
\R^d$ be a nonempty compact set, $h:(0, \, 1] \rightarrow (0, \,
\infty)$ be a non-decreasing function. We say that $\Gamma$ is an
$h$-set if there is a constant $c_*\ge 1$ and a finite
countable-additive measure $\mu$ on $\R^d$ such that $\supp
\mu=\Gamma$ and
$$
c_*^{-1}h(t)\le \mu(B_t(x))\le c_* h(t)
$$
for all $x\in \Gamma$ and $t\in (0, \, 1]$.
\end{Def}

Examples of $h$-sets are Lipschitz manifolds of dimension $k$
(then $h(t)=t^k$), some Cantor-type sets, the Koch's curve.

In order to formulate the main results we need
\begin{Def}
\label{theta_j} Given $s_*$, $\tilde \theta$, $\hat \theta\in \R$,
we define the numbers $j_0\in \N$ and $\theta_j\in \R$ $(1\le j\le
j_0)$ as follows.
\begin{enumerate}
\item For $p_0\ge q$, $p_1\ge q$: $j_0=2$, $\theta_1=s_*$, $\theta_2 = \tilde
\theta$.
\item For $p_0>q$, $p_1<q\le 2$: $j_0=3$,
$\theta_1=s_*+\frac 1q-\frac{1}{p_1}$, $\theta_2 = \tilde \theta$,
$\theta_3 =\hat \theta$.
\item For $p_0>q$, $2\le p_1<q$: $j_0=4$, $\theta_1=s_*$, $\theta_2 =
\frac{q(s_*+1/q-1/p_1)}{2}$, $\theta_3 =\tilde \theta$, $\theta_4=
\frac{q\hat \theta}{2}$.
\item For $p_0>q$, $p_1<2<q$: $j_0=5$, $\theta_1=s_* +\frac 12 -\frac
{1}{p_1}$, $\theta_2 =\frac{q(s_*+1/q-1/p_1)}{2}$,
$\theta_3=\tilde \theta$, $\theta_4 = \hat\theta + \frac 12-\frac
1q$, $\theta_5 = \frac{q\hat \theta}{2}$.
\item For $p_0\le q$, $p_1\le q\le 2$: $j_0=2$,
$\theta_1= s_*+\frac 1q-\frac {1}{p_1}$, $\theta_2 = \hat \theta$.
\item For $p_0<q\le 2$, $p_1>q$: $j_0=3$, $\theta_1 = s_*$,
$\theta_2=\tilde \theta$, $\theta_3 =\hat \theta$.
\item For $p_0<q$, $q>2$, $\max\{p_0, \, p_1\}\le 2$: $j_0=4$, $\theta_1 =s_* + \frac 12 -\frac
{1}{p_1}$, $\theta_2 =\frac{q(s_*+1/q-1/p_1)}{2}$, $\theta_3 =\hat
\theta +\frac 12 -\frac 1q$, $\theta_4 =\frac{q\hat \theta}{2}$.
\item For $p_0<q$, $q>2$, $\min\{p_0, \, p_1\}\ge 2$: $j_0=4$, $\theta_1 =s_*$, $\theta_2 =\frac{q(s_*+1/q-1/p_1)}{2}$,
$\theta_3 =\tilde \theta$, $\theta_4 =\frac{q\hat \theta}{2}$.
\item For $p_0<q$, $q>2$, $\min \{p_0, \, p_1\}<2< \max \{p_0, \,
p_1\}$: $j_0=5$, $\theta_1=s_* +\min\left\{\frac 12 -\frac
{1}{p_1}, \, 0\right\}$, $\theta_2 =\frac{q(s_*+1/q-1/p_1)}{2}$,
$\theta_3 =\tilde \theta$, $\theta_4 = \hat \theta +\frac 12-\frac
1q$, $\theta_5 =\frac{q\hat\theta}{2}$.
\end{enumerate}
\end{Def}

Given $d\in \N$, $r\in \N$, we set
\begin{align}
\label{s_st} s_*=\frac rd.
\end{align}

Consider the first example.

Let $\Omega \subset \left(-\frac 12, \, \frac 12\right)^d$,
$\Omega \in {\bf FC}(a)$, $\Gamma \subset\partial \Omega$ be an
$h$-set,
\begin{align}
\label{h_theta} h(t) = t^\theta,
\end{align}
$0\le \theta <d$, $r\in \N$, $1<p_0, \, p_1\le \infty$,
$1<q<\infty$, $\beta$, $\lambda$, $\sigma\in \R$,
\begin{align}
\label{gw} g(x) ={\rm \, \dist}^{-\beta}(x, \, \Gamma), \quad w(x)
={\rm dist}^{-\sigma}(x, \, \Gamma), \quad v(x) ={\rm \,
\dist}^{-\lambda}(x, \, \Gamma).
\end{align}

We denote $\mathfrak{Z}=(r, \, d, \, p_0, \, p_1, \, q, \, a, \,
c_*, \theta, \, \beta, \, \lambda, \, \sigma)$,
$\mathfrak{Z}_*=(\mathfrak{Z}, \, R)$, where $R={\rm diam}\,
\Omega$, $c_*$ is from Definition \ref{h_set}.

We set
\begin{align}
\label{til_th} \tilde \theta =\frac rd \cdot \frac{\sigma -\lambda
+ \frac{d-\theta}{q} -\frac{d-\theta}{p_0}}{\beta +\sigma
-\left(r+\frac{d}{p_0} -\frac{d}{p_1}\right)\left(1-\frac
{\theta}{d}\right)},
\end{align}
\begin{align}
\label{hat_th} \hat \theta =\frac{\sigma \left(\frac rd +\frac
1q-\frac{1}{p_1}\right) +\beta\left(\frac 1q -\frac{1}{p_0}\right)
-\lambda\left(\frac rd +\frac{1}{p_0} -\frac{1}{p_1}\right)}{\beta
+\sigma -\left(r+\frac{d}{p_0} -\frac{d}{p_1}\right)\left(1-\frac
{\theta}{d}\right)}.
\end{align}

We use the following notation for order equalities and
inequalities. Let $X$, $Y$ be sets, $f_1$, $f_2:\ X\times
Y\rightarrow \mathbb{R}_+$. We write $f_1(x, \,
y)\underset{y}{\lesssim} f_2(x, \, y)$ (or $f_2(x, \,
y)\underset{y}{\gtrsim} f_1(x, \, y)$) if for each $y\in Y$ there
exists $c(y)>0$ such that $f_1(x, \, y)\le c(y)f_2(x, \, y)$ for
all $x\in X$; $f_1(x, \, y)\underset{y}{\asymp} f_2(x, \, y)$ if
$f_1(x, \, y) \underset{y}{\lesssim} f_2(x, \, y)$ and $f_2(x, \,
y)\underset{y}{\lesssim} f_1(x, \, y)$.

\begin{Trm}
\label{main1} Let (\ref{s_st}), (\ref{h_theta}), (\ref{gw}),
(\ref{til_th}), (\ref{hat_th}) hold, $\frac rd +\min \left\{\frac
1q, \, \frac{1}{p_0}\right\}-\frac{1}{p_1}>0$, $$\min\left\{\beta
+ \sigma - r -\frac{d-\theta}{p_0} +\frac{d-\theta}{p_1}, \, \beta
+\sigma - r-\frac{d}{p_0} +\frac{d}{p_1}\right\}>0;$$ let $\tilde
\theta>0$ for $p_0\ge q$, $\hat \theta>0$ for $p_0<q$. The set $M$
is defined by (\ref{m_def}), the numbers $j_0\in \N$ and
$\theta_j$ are as in Definition \ref{theta_j}. Suppose that there
exists $j_*\in \{1, \, \dots, \, j_0\}$ such that $\theta_{j_*}<
\min _{j\ne j_*} \theta_j$. Then
$$
d_n(M, \, L_{q,v}(\Omega)) \underset{\mathfrak{Z}_*}{\asymp}
n^{-\theta_{j_*}}.
$$
\end{Trm}

Consider the second example.

Let $\Omega \subset \left(-\frac 12, \, \frac 12\right)^d$,
$\Omega \in {\bf FC}(a)$, let $\Gamma \subset \partial \Omega$ be
an $h$-set,
\begin{align}
\label{gvw} g(x)=\varphi_g({\rm dist}\, (x, \, \Gamma)), \quad
w(x) =\varphi_w({\rm dist}\, (x, \, \Gamma)), \quad v(x)
=\varphi_v({\rm dist}\, (x, \, \Gamma)),
\end{align}
\begin{align}
\label{h_gam} h(t) = |\log t|^{-\gamma}, \quad \gamma \ge 0,
\end{align}
\begin{align}
\label{gw_1} \varphi_g(t) = t^{-\beta} |\log t|^{\mu}, \quad
\varphi_w(t) = t^{-\sigma} |\log t|^{\alpha}, \quad \varphi_v(t) =
t^{-\lambda} |\log t|^{\nu},
\end{align}
where
\begin{align}
\label{lim_case} \beta + \lambda = r +\frac dq -\frac{d}{p_1},
\quad \sigma -\lambda = \frac{d}{p_0} - \frac{d}{q}.
\end{align}

We set
\begin{align}
\label{til_th2} \tilde \theta = \frac rd \cdot \frac{\alpha -\nu +
(\gamma+1)\left(\frac{1}{p_0}-\frac{1}{q}\right)}{\mu+\alpha
+(\gamma+1)\left(\frac rd+\frac{1}{p_0}-\frac{1}{p_1}\right)},
\end{align}
\begin{align}
\label{hat_th2} \hat \theta = \frac{\alpha\left(\frac{r}{d}+\frac
1q-\frac{1}{p_1}\right)+\mu\left(\frac 1q-\frac{1}{p_0}\right
)-\nu\left(\frac rd +\frac{1}{p_0}- \frac{1}{p_1}\right)}
{\mu+\alpha+(\gamma+1)\left(\frac
rd+\frac{1}{p_0}-\frac{1}{p_1}\right)}.
\end{align}

We denote $\mathfrak{Z}=(r, \, d, \, p_0, \, p_1, \, q, \, a, \,
c_*, h, \, \varphi_g, \, \varphi_w, \, \varphi_v)$,
$\mathfrak{Z}_*=(\mathfrak{Z}, \, R)$, where $R={\rm diam}\,
\Omega$.

\begin{Trm}
\label{main2} Let $\frac rd +\min\left\{\frac 1q, \,
\frac{1}{p_0}\right\} -\frac{1}{p_1}>0$, let (\ref{s_st}),
(\ref{gvw}), (\ref{h_gam}), (\ref{gw_1}), (\ref{lim_case}),
(\ref{til_th2}), (\ref{hat_th2}) hold, $\min \{\mu+\alpha
+(\gamma+1)(1/p_0-1/p_1), \, \mu + \alpha\}>0$. Suppose that
$\tilde \theta>0$ for $p_0\ge q$, $\hat\theta>0$ for $p_0<q$. The
set $M$ is defined by (\ref{m_def}), $j_0\in \N$ and $\theta_j$
are as in Definition \ref{theta_j}. Suppose that there exists
$j_*\in \{1, \, \dots, \, j_0\}$ such that $\theta_{j_*}< \min
_{j\ne j_*} \theta_j$. Then
$$
d_n(M, \, L_{q,v}(\Omega)) \underset{\mathfrak{Z}_*}{\asymp}
n^{-\theta_{j_*}}.
$$
\end{Trm}

Consider the third example.

Let $\Omega = \R^d$,
\begin{align}
\label{gw_2} g(x)=(1+|x|)^\beta, \quad w(x) = (1+|x|)^\sigma,
\quad v(x) = (1 + |x|)^\lambda.
\end{align}
We set
\begin{align}
\label{til_theta3} \tilde \theta = \frac rd \cdot \frac{\sigma
-\lambda + \frac{d}{p_0}-\frac{d}{q}}{\beta +\sigma
+r+\frac{d}{p_0} - \frac{d}{p_1}},
\end{align}
\begin{align}
\label{hat_theta3} \hat \theta =
\frac{\sigma\left(\frac{r}{d}+\frac 1q- \frac{1}{p_1}\right)
+\beta(\frac 1q- \frac{1}{p_0}) -\lambda \left(\frac rd
+\frac{1}{p_0} -\frac{1}{p_1}\right)}{\beta +\sigma
+r+\frac{d}{p_0} - \frac{d}{p_1}}.
\end{align}

Denote $\mathfrak{Z}=(r, \, d, \, p_0, \, p_1, \, q, \, \beta, \,
\lambda, \, \sigma)$.

\begin{Trm}
\label{main3} Let (\ref{s_st}), (\ref{gw_2}), (\ref{til_theta3}),
(\ref{hat_theta3}) hold, and let $\frac rd+\min\left\{\frac 1q, \,
\frac{1}{p_0}\right\}-\frac{1}{p_1}>0$, $\beta+\sigma+r
+d/p_0-d/p_1>0$. Suppose that $\tilde \theta>0$ for $p_0\ge q$,
$\hat\theta>0$ for $p_0<q$. The set $M$ is defined by
(\ref{m_def}), $j_0\in \N$ and $\theta_j$ are as in Definition
\ref{theta_j}. Suppose that there is $j_*\in \{1, \, \dots, \,
j_0\}$ such that $\theta_{j_*}< \min _{j\ne j_*} \theta_j$. Then
$$
d_n(M, \, L_{q,v}(\R^d)) \underset{\mathfrak{Z}}{\asymp}
n^{-\theta_{j_*}}.
$$
\end{Trm}
Notice that in \cite{myn_otel} the problem on estimating the
linear widths of the set $M$ from the third example was considered
for $p_0=p_1$; for $q\le 2$ and $\frac{p_1}{p_1-1}\le 2$ order
estimates for the linear widths were obtained.

The paper is organized as follows. In \S 2 the upper estimate for
the Kolmogorov widths of the abstract function classes
$BX_{p_1}(\Omega)\cap BX_{p_0}(\Omega)$ is obtained; in \S 3 the
lower estimate is obtained. In \S 4--5 these results are applied
for proofs of Theorems \ref{main1}--\ref{main3}.

\section{Upper estimates for widths of $BX_{p_1}(\Omega)\cap BX_{p_0}(\Omega)$.}

Let $(\Omega, \, \Sigma, \, {\rm mes})$ be a measure space. We say
that sets $A$, $B\subset \Omega$ are disjoint if ${\rm mes}(A\cap
B)=0$. Let $E$, $E_1, \, \dots, \, E_m\subset \Omega$ be
measurable sets, $m\in \N\cup \{\infty\}$. We say that
$\{E_i\}_{i=1}^m$ is a partition of $E$ if the sets $E_i$ are
disjoint and ${\rm mes}\left(\left(\cup _{i=1}^m
E_i\right)\bigtriangleup E\right)=0$.

We denote by $\chi_E(\cdot)$ the indicator function of $E$.

Let $1< p_0, \, p_1\le \infty$, $1\le q< \infty$. Suppose that for
each measurable subset $E\subset \Omega$ the following spaces are
defined (see \cite{vas_width_raspr}):
\begin{itemize}
\item the spaces $X_{p_i}(E)$ with seminorms
$\|\cdot\|_{X_{p_i}(E)}$, $i=0, \, 1$,
\item the Banach space $Y_q(E)$ with norm $\|\cdot\|_{Y_q(E)}$,
\end{itemize}
which satisfy the following conditions:
\begin{enumerate}
\item $X_{p_i}(E)=\{f|_E:\; f\in X_{p_i}(\Omega)\}$, $i=0, \, 1$, $Y_q(E)=\{f|_E:\; f\in
Y_q(\Omega)\}$;
\item if ${\rm mes}\, E=0$, then $\dim \, Y_q(E)=\dim \, X_{p_i}(E)=0$, $i=0, \, 1$;
\item if $E\subset \Omega$, $E_j\subset \Omega$ ($j\in \N$)
are measurable subsets, $E=\sqcup _{j\in \N} E_j$, then
$$
\|f\|_{X_{p_i}(E)}=\left\| \bigl\{
\|f|_{E_j}\|_{X_{p_i}(E_j)}\bigr\}_{j\in
\N}\right\|_{l_{p_i}},\quad f\in X_{p_i}(E), \; i=0, \, 1,
$$
$$
\|f\|_{Y_q(E)}=\left\| \bigl\{\|f|_{E_j}\|
_{Y_q(E_j)}\bigr\}_{j\in \N}\right\|_{l_q}, \quad f\in Y_q(E);
$$
\item if $E\in \Sigma$, $f\in Y_q(\Omega)$, then $f\cdot \chi_E\in
Y_q(\Omega)$.
\end{enumerate}

We denote $$BX_{p_i}(\Omega) = \{f\in X_{p_i}(\Omega):\;
\|f\|_{X_{p_i}(\Omega)}\le 1\}, \quad i=0, \, 1.$$

Let ${\cal P}(\Omega)\subset X_{p_1}(\Omega)$  be a subspace of
dimension $r_0\in \N$. For each measurable subset $E\subset
\Omega$ we denote
$${\cal P}(E)=\{P|_E:\; P\in {\cal P}(\Omega)\}.$$ Let $G\subset
\Omega$ be a measurable subset, and let $T$ be a partition of $G$.
We set
$$
{\cal S}_{T}(\Omega)=\{f:\Omega\rightarrow \R:\, f|_E\in {\cal
P}(E),\; E\in T, \; f|_{\Omega\backslash G}=0\}.
$$
If $T$ is finite and for each $E\in T$ the inclusion ${\cal
P}(E)\subset Y_q(E)$ holds, then ${\cal S}_{T}(\Omega)\subset
Y_q(\Omega)$ (see property 4).

For each finite partition $T=\{E_j\}_{j=1}^n$ of a set $E$ and for
each function $f\in Y_q(\Omega)$ we set
$$
\|f\|_{p_i,q,T}=\left(\sum \limits _{j=1}^n
\|f|_{E_j}\|_{Y_q(E_j)} ^{p_i}\right)^{\frac{1}{p_i}}.
$$

Suppose that there exist a partition $\{\Omega _{t,j}\}_{t\ge t_0,
\, j\in \hat J_t}$ of $\Omega$ into measurable subsets ($t_0\in
\Z_+$) and numbers $c\ge 1$, $s_*>\left(\frac{1}{p_1} -\frac
1q\right)_+$, $k_*\in \N$, $\gamma_*\ge 0$, $\alpha_*\in \R$,
$\mu_*\in \R$ such that the following assumptions hold.

\begin{Supp} \label{supp1}
The inclusion $X_{p_1}(\Omega_{t,j})\subset Y_q(\Omega_{t,j})$
holds for each $t\ge t_0$, $j\in \hat J_t$.
\end{Supp}
\begin{Supp} \label{supp2}
The following estimate holds:
\begin{align}
\label{card_jt} {\rm card}\, \hat J_t\le c\cdot 2^{\gamma_*k_*t},
\quad t\ge t_0.
\end{align}
\end{Supp}
\begin{Supp} \label{supp3}
For each $t\ge t_0$, $j\in \hat J_t$ there is a sequence of
partitions $\{T_{t,j,m}\}_{m\in \Z_+}$ of the set $\Omega_{t,j}$
such that
\begin{align}
\label{ttj0} T_{t,j,0} = \{\Omega_{t,j}\}, \quad {\rm card}\,
T_{t,j,m}\le c\cdot 2^m,
\end{align}
and for all $E\in T_{t,j,m}$
\begin{align}
\label{card_e} {\rm card}\, \{E'\in T_{t,j,m\pm 1}:\; {\rm mes}\,
(E\cap E')>0\} \le c.
\end{align}
\end{Supp}
\begin{Supp} \label{supp4}
If $p_0\ge q$, then for each $E\in T_{t,j,m}$
\begin{align}
\label{f_yqe} \|f\|_{Y_q(E)}\le c\cdot 2^{-\alpha_*k_*t}\cdot
2^{m\left(\frac{1}{p_0}-\frac 1q\right)}\|f\|_{X_{p_0}(E)}.
\end{align}
\end{Supp}
\begin{Supp} \label{supp5}
For each $E\in T_{t,j,m}$ there is a linear continuous projection
$P_E:Y_q(\Omega) \rightarrow {\cal P}(\Omega)$ such that for all
$f\in X_{p_1}(\Omega)\cap X_{p_0}(\Omega)$
\begin{align}
\label{fpef} \|f-P_Ef\|_{Y_q(E)}\le c \cdot 2^{\mu_*k_*t}\cdot
2^{-m\left(s_*+\frac 1q-\frac{1}{p_1}\right)} \|f\| _{X_{p_1}(E)},
\end{align}
\begin{align}
\label{pef} \|P_Ef\|_{Y_q(E)}\le c\cdot 2^{-\alpha_*k_*t}\cdot
2^{m\left(\frac{1}{p_0}-\frac 1q\right)}\|f\|_{X_{p_0}(E)}.
\end{align}
\end{Supp}

Let
$$
\mathfrak{Z}_0=(p_0, \, p_1, \, q, \, r_0, \, c, \, k_*,\, s_*, \,
\gamma_*, \, \mu_*, \, \alpha_*).
$$

We define the partitions
$$
T_{t,m} = \{E\in T_{t,j,m}:\; j\in \hat J_t\}, \quad \hat T_{t,m}
= \{E\cap E':\;  E\in T_{t,m}, \; E'\in T_{t,m+1}\}.
$$
Then
\begin{align}
\label{ttm} {\rm card}\, T_{t,m} \stackrel{(\ref{card_jt}),
(\ref{ttj0})}{\underset{\mathfrak{Z}_0}{\lesssim}} 2^{\gamma
_*k_*t}\cdot 2^m, \quad {\rm card}\, \hat T_{t,m}
\stackrel{(\ref{card_e})}{\underset{\mathfrak{Z}_0}{\lesssim}}
2^{\gamma _*k_*t}\cdot 2^m.
\end{align}

Let $$\nu'_{t,m} = \dim\, {\cal S}_{ T_{t,m}}(\Omega), \quad
\nu_{t,m} = \dim\, {\cal S}_{\hat T_{t,m}}(\Omega).$$

We define the operator $P_{t,m}:Y_q(\Omega)\rightarrow
Y_q(\Omega)$ by
$$
P_{t,m}f = \sum \limits _{j\in \hat J_t} \sum \limits _{E\in
T_{t,j,m}} P_Ef \cdot \chi _E.
$$
Then ${\rm Im}\, P_{t,m} \subset {\cal S}_{T_{t,m}}(\Omega)$,
${\rm Im}\, (P_{t,m+1}-P_{t,m})\subset {\cal S}_{\hat
T_{t,m}}(\Omega)$,
\begin{align}
\label{rk_ptm} {\rm rk}\, P_{t,m}\le \nu'_{t,m}
\stackrel{(\ref{ttm})}{\underset{\mathfrak{Z}_0}{\lesssim}}
2^{\gamma _*k_*t}\cdot 2^m, \quad {\rm rk}\,
(P_{t,m+1}-P_{t,m})\le \nu_{t,m}
\stackrel{(\ref{ttm})}{\underset{\mathfrak{Z}_0}{\lesssim}}
2^{\gamma _*k_*t}\cdot 2^m.
\end{align}

We set $\Omega_t=\cup _{j\in \hat J_t} \Omega _{t,j}$.

Applying the H\"{o}lder's inequality and the inequality
$$
\left(\sum \limits _{i=1}^k |x_i|^q\right)^{1/q} \le \left(\sum
\limits _{i=1}^k |x_i|^p\right)^{1/p}, \quad (x_i)_{i=1}^k\in
\R^k, \quad p\le q,
$$
we obtain that for $f\in BX_{p_1}(\Omega)\cap BX_{p_0}(\Omega)$
the following estimates hold:
\begin{align}
\label{f_m_ptm} \|f-P_{t,m}f\|_{Y_q(\Omega_t)}
\stackrel{(\ref{fpef}),
(\ref{ttm})}{\underset{\mathfrak{Z}_0}{\lesssim}}
2^{k_*t\left(\mu_*+\gamma_*(1/q-1/p_1)_+\right)}\cdot
2^{-m(s_*-(1/p_1-1/q)_+)},
\end{align}
\begin{align}
\label{ptm1f_m_ptmf} \|P_{t,m+1}f - P_{t,m}f\| _{p_1,q,\hat
T_{t,m}} \stackrel{(\ref{card_e}),(\ref{fpef})}
{\underset{\mathfrak{Z}_0}{\lesssim}} 2^{\mu_*k_*t}\cdot
2^{-m(s_*+1/q-1/p_1)},
\end{align}
\begin{align}
\label{ptm_p0} \|P_{t,m}f\|_{p_0,q,T_{t,m}}
\stackrel{(\ref{pef})}{\underset{\mathfrak{Z}_0}{\lesssim}}
2^{-\alpha_*k_*t}\cdot 2^{m(1/p_0-1/q)},
\end{align}
\begin{align}
\label{ptm_p011111} \|P_{t,m}f\|_{Y_q(\Omega_t)}
\stackrel{(\ref{pef})}{\underset{\mathfrak{Z}_0}{\lesssim}}
2^{-\alpha_*k_*t}\cdot 2^{m(1/p_0-1/q)}, \quad \text{if} \quad
p_0\le q,
\end{align}
\begin{align}
\label{ptm1_m_ptm_p0} \|P_{t,m+1}f - P_{t,m}f\| _{p_0,q,\hat
T_{t,m}} \stackrel{(\ref{card_e}),
(\ref{pef})}{\underset{\mathfrak{Z}_0}{\lesssim}}
2^{-\alpha_*k_*t}\cdot 2^{m(1/p_0-1/q)}.
\end{align}

We denote by $l_p^\nu$ the space $\R^\nu$ with the norm $\|(x_1,
\, \dots, \, x_\nu)\|_{l_p^\nu}=\left(\sum \limits _{j=1}^\nu
|x_j|^p\right)^{1/p}$. By $B_p^\nu$ we denote the unit ball in
$l_p^\nu$.

\begin{Sta}
There exist isomorphisms $A_{t,m}: {\cal S}_{\hat T_{t,m}}(\Omega)
\rightarrow \R^{\nu_{t,m}}$ and $A'_{t,m}: {\cal S}_{
T_{t,m}}(\Omega) \rightarrow \R^{\nu'_{t,m}}$ such that
\begin{align}
\label{atm} \|A_{t,m}f\|_{l_{p_1}^{\nu_{t,m}}}
\underset{\mathfrak{Z}_0}{\lesssim} \|f\|_{p_1,q,\hat T_{t,m}},
\quad \|A_{t,m}f\|_{l_{p_0}^{\nu_{t,m}}}
\underset{\mathfrak{Z}_0}{\lesssim} \|f\|_{p_0,q,\hat T_{t,m}},
\quad f\in {\cal S}_{\hat T_{t,m}}(\Omega),
\end{align}
\begin{align}
\label{atm1} \|A_{t,m}^{-1}(c_j)_{j=1}^{\nu_{t,m}}\|
_{Y_q(\Omega)} \underset{\mathfrak{Z}_0}{\lesssim}
\|(c_j)_{j=1}^{\nu_{t,m}}\| _{l_q^{\nu_{t,m}}},
\end{align}
\begin{align}
\label{atm2} \|A'_{t,m}f\|_{l_{p_0}^{\nu'_{t,m}}}
\underset{\mathfrak{Z}_0}{\lesssim} \|f\|_{p_0,q,T_{t,m}}, \quad
f\in {\cal S}_{T_{t,m}}(\Omega), \quad
\|(A'_{t,m})^{-1}(c_j)_{j=1}^{\nu'_{t,m}}\| _{Y_q(\Omega)}
\underset{\mathfrak{Z}_0}{\lesssim} \|(c_j)_{j=1}^{\nu'_{t,m}}\|
_{l_q^{\nu'_{t,m}}}.
\end{align}
\end{Sta}
\begin{proof}
Let $E \in \hat T_{t,m}$, ${\rm mes}(E)>0$, $\nu_E =\dim {\cal
P}(E)$. Then
\begin{align}
\label{nu_e} \nu_E\le r_0.
\end{align}
By John's ellipsoid theorem, there is an isomorphism $A_E: {\cal
P}(E)\rightarrow l_2^{\nu_E}$ such that
\begin{align}
\label{a_e} \|A_Ef\|_{l_2^{\nu_E}}\le \|f\|_{Y_q(E)}\le
\sqrt{\nu_E} \|A_Ef\|_{l_2^{\nu_E}}, \quad f\in {\cal P}(E).
\end{align}
We set for $f\in {\cal S}_{\hat T_{t,m}}(\Omega)$
$$
A_{t,m}f =(A_E(f|_E))_{E\in \hat T_{t,m}}.
$$
Then from (\ref{nu_e}) and (\ref{a_e}) we get (\ref{atm}),
(\ref{atm1}). The isomorphism $A'_{t,m}$ satisfying (\ref{atm2})
can be constructed similarly.
\end{proof}

Let $W_{t,m}$ be the set of sequences $(c_j)_{j=1} ^{\nu_{t,m}}\in
\R^{\nu_{t,m}}$ such that
\begin{align}
\label{w_tm} \begin{array}{c} \left(\sum \limits
_{j=1}^{\nu_{t,m}}|c_j|^{p_1} \right)^{1/p_1} \le
2^{\mu_*k_*t}\cdot 2^{-m(s_*+1/q-1/p_1)}, \\ \left(\sum \limits
_{j=1}^{\nu_{t,m}}|c_j|^{p_0} \right)^{1/p_0} \le
2^{-\alpha_*k_*t}\cdot 2^{m(1/p_0-1/q)}.\end{array}
\end{align}

\begin{Sta}
Let $l\in \Z_+$. Then
\begin{align}
\label{dl} d_l((P_{t,m+1}-P_{t,m})(BX_{p_1}(\Omega)\cap
BX_{p_0}(\Omega)), \, Y_q(\Omega))
\underset{\mathfrak{Z}_0}{\lesssim} d_l(W_{t,m}, \,
l_q^{\nu_{t,m}}),
\end{align}
\begin{align}
\label{dl1} d_l(P_{t,m}(BX_{p_1}(\Omega)\cap BX_{p_0}(\Omega)), \,
Y_q(\Omega)) \underset{\mathfrak{Z}_0}{\lesssim}
2^{-\alpha_*k_*t}\cdot 2^{-m(1/q-1/p_0)}d_l(B_{p_0}^{\nu'_{t,m}},
\, l_q^{\nu'_{t,m}}).
\end{align}
\end{Sta}
\begin{proof}
Let $L\subset l_q^{\nu_{t,m}}$ be an extremal subspace for the
widths $d_l(W_{t,m}, \, l_q^{\nu_{t,m}})$, let
$E_{t,m}:l_q^{\nu_{t,m}} \rightarrow L$ be the metric projection,
and let  $I_{t,m}: l_q^{\nu_{t,m}} \rightarrow l_q^{\nu_{t,m}}$ be
the identity operator. From (\ref{ptm1f_m_ptmf}),
(\ref{ptm1_m_ptm_p0}) and (\ref{atm}) it follows that
$$A_{t,m}(P_{t,m+1}-P_{t,m})(BX_{p_1}(\Omega)\cap BX_{p_0}(\Omega))
\subset \hat c(\mathfrak{Z}_0) W_{t,m}$$ for some positive
constant $\hat c(\mathfrak{Z}_0)$. Hence,
$$
d_l((P_{t,m+1}-P_{t,m})(BX_{p_1}(\Omega)\cap BX_{p_0}(\Omega)), \,
Y_q(\Omega)) \underset{\mathfrak{Z}_0}{\lesssim}$$$$ \lesssim
d_l(A_{t,m}^{-1}A_{t,m}(P_{t,m+1}-P_{t,m})(BX_{p_1}(\Omega)\cap
BX_{p_0}(\Omega)), \, Y_q(\Omega))
\stackrel{(\ref{atm1})}{\underset{\mathfrak{Z}_0}{\lesssim}}
$$
$$
\lesssim d_l(A_{t,m}(P_{t,m+1}-P_{t,m})(BX_{p_1}(\Omega)\cap
BX_{p_0}(\Omega)), \, l_q^{\nu_{t,m}}) \le$$$$\le
\|(I_{t,m}-E_{t,m})A_{t,m}(P_{t,m+1}-P_{t,m})(BX_{p_1}(\Omega)\cap
BX_{p_0}(\Omega))\|_{l_q^{\nu_{t,m}}}\underset{\mathfrak{Z}_0}{\lesssim}
$$
$$
\lesssim\|(I_{t,m}-E_{t,m})W_{t,m}\|_{l_{q}^{\nu_{t,m}}} =
d_l(W_{t,m}, \, l_q^{\nu_{t,m}}).
$$

The inequality (\ref{dl1}) can be proved similarly applying
(\ref{ptm_p0}) and (\ref{atm2}).
\end{proof}

Let $k_0>0$, $k_1>0$, $\nu\in \N$, $1< \tilde q < \infty$,
$\frac{1}{\tilde q} =\frac{1-\lambda}{p_1}+\frac{\lambda}{p_0}$,
$\lambda\in (0, \, 1)$. Then the H\"{o}lder's inequality yields
the inclusion
$$
k_0B_{p_0}^\nu\cap k_1B_{p_1}^\nu \subset
k_0^{\lambda}k_1^{1-\lambda} B_{\tilde q}^\nu
$$
(it is also the particular case of Gallev's result \cite[Theorem
2]{galeev1}).

If $\frac 1q = \frac{1-\lambda}{p_1} + \frac{\lambda}{p_0}$ with
$\lambda \in (0, \, 1)$, then by (\ref{w_tm}) and the equality
\begin{align}
\label{1lams}
(1-\lambda)(s_*+1/q-1/p_1)+\lambda(1/q-1/p_0)=(1-\lambda)s_*
\end{align}
we obtain
\begin{align}
\label{wkm_lq} W_{k,m} \subset
2^{k_*t\left((1-\lambda)\mu_*-\lambda \alpha_*\right)} \cdot
2^{-ms_*(1-\lambda)}B_q^{\nu_{t,m}},
\end{align}
$$
d_l((P_{t,m+1}-P_{t,m})(BX_{p_1}(\Omega)\cap BX_{p_0}(\Omega)), \,
Y_q(\Omega))
\stackrel{(\ref{dl})}{\underset{\mathfrak{Z}_0}{\lesssim}}
2^{k_*t\left((1-\lambda)\mu_*-\lambda \alpha_*\right)} \cdot
2^{-ms_*(1-\lambda)};
$$
in particular, for $l=0$ we get
\begin{align}
\label{wkm_lqd0} \|P_{t,m+1}f-P_{t,m}f\|
_{Y_q(\Omega_t)}\underset{\mathfrak{Z}_0}{\lesssim}
2^{k_*t\left((1-\lambda)\mu_*-\lambda \alpha_*\right)} \cdot
2^{-ms_*(1-\lambda)}, \quad f\in BX_{p_1}(\Omega)\cap
BX_{p_0}(\Omega).
\end{align}

If $\frac 12 = \frac{1-\tilde\lambda}{p_1} +
\frac{\tilde\lambda}{p_0}$ with $\tilde\lambda \in (0, \, 1)$,
then
\begin{align}
\label{wkm_l2} W_{k,m} \subset 2^{k_*t\left( (1-\tilde\lambda)
\mu_*-\tilde\lambda \alpha_*\right)} \cdot
2^{-m((1-\tilde\lambda)(s_*+1/q-1/p_1)+\tilde \lambda
(1/q-1/p_0))}B_2^{\nu_{t,m}},
\end{align}
$$
\begin{array}{c} d_l((P_{t,m+1}-P_{t,m})(BX_{p_1}(\Omega)\cap
BX_{p_0}(\Omega)), \, Y_q(\Omega))
\underset{\mathfrak{Z}_0}{\lesssim} \\ \lesssim 2^{k_*t\left(
(1-\tilde\lambda) \mu_*-\tilde\lambda \alpha_*\right)} \cdot
2^{-m((1-\tilde\lambda)(s_*+1/q-1/p_1)+\tilde \lambda
(1/q-1/p_0))} d_l(B_2^{\nu_{t,m}}, \, l_q^{\nu_{t,m}}).
\end{array}
$$

We denote $\tilde \Omega_t = \cup _{l\ge t} \Omega_l$.

\begin{Sta}
\label{emb1} Let $p_0\ge q$, $\frac{\alpha_*}{\gamma_*}>
\frac{1}{q} -\frac{1}{p_0}$. The for each $t\ge t_0$ and for each
function $f\in BX_{p_0}(\Omega)\cap BX_{p_1}(\Omega)$
\begin{align}
\label{p0geq_emb} \|f\|_{Y_q(\tilde\Omega_t)}
\underset{\mathfrak{Z}_0}{\lesssim}
2^{-(\alpha_*-\gamma_*/q+\gamma_*/p_0)k_*t} .
\end{align}
\end{Sta}
\begin{proof}
By H\"{o}lder's inequality,
$$
\|f\|_{Y_q(\tilde\Omega_t)}^q = \sum \limits _{l\ge t} \sum
\limits _{j\in \hat J_l} \|f\|^q_{Y_q(\Omega_{l,j})}
\stackrel{(\ref{ttj0}),(\ref{f_yqe})}{\underset{\mathfrak{Z}_0}{\lesssim}}
\sum \limits _{l\ge t} \sum \limits _{j\in \hat J_l}
2^{-q\alpha_*l} \|f\|^q _{X_{p_0}(\Omega_{l,j})}
\stackrel{(\ref{card_jt})}{\underset{\mathfrak{Z}_0}{\lesssim}}
$$
$$
\lesssim \sum \limits _{l\ge t}
2^{\gamma_*\left(1-\frac{q}{p_0}\right)k_*l}\cdot 2^{-q\alpha_*l}
\|f\|^q _{X_{p_0}(\Omega_{l})} \underset{\mathfrak{Z}_0}{\lesssim}
2^{-q(\alpha_*-\gamma_*/q+\gamma_*/p_0)k_*t}.
$$
\end{proof}
\begin{Sta}
\label{emb2} Let $p_0\le q$ and $p_1\le q$, or $p_0<q$, $p_1>q$.
Suppose that $\alpha_*+\mu_*>0$, $\alpha_*\left(s_*+\frac
1q-\frac{1}{p_1}\right)>\mu_*\left(\frac{1}{p_0}-\frac 1q\right)$.
Then for each $t\ge t_0$ and for each function $f\in
BX_{p_0}(\Omega)\cap BX_{p_1}(\Omega)$
\begin{align}
\label{p0lq_emb} \|f\|_{Y_q(\tilde\Omega_t)}
\underset{\mathfrak{Z}_0}{\lesssim}
2^{-\frac{\alpha_*(s_*+1/q-1/p_1)-\mu_*(1/p_0-1/q)}
{s_*+1/p_0-1/p_1}k_*t}.
\end{align}
\end{Sta}
\begin{proof}
If $p_0=q$, (\ref{p0lq_emb}) follows from (\ref{p0geq_emb}).
Further we assume that $p_0<q$.

We define $m_l\in \R$ by
\begin{align}
\label{mt_def} 2^{\mu_*k_*l}\cdot 2^{-\left(s_*+\frac 1q
-\frac{1}{p_1}\right)m_l} = 2^{-\alpha_*k_*l}\cdot
2^{\left(\frac{1}{p_0}-\frac 1q\right)m_l}.
\end{align}
Then
\begin{align}
\label{mt_form} 2^{\left(s_*+\frac{1}{p_0}- \frac{1}{p_1}\right)
m_l} = 2^{(\mu_*+\alpha_*)k_*l}.
\end{align}
Notice that $s_*+\frac{1}{p_0} -\frac{1}{p_1} = \left(s_* + \frac
1q -\frac{1}{p_1}\right) +\left(\frac{1}{p_0} -\frac 1q\right)>0$.

We claim that
\begin{align}
\label{sum_emb} \sum \limits _{l\ge t} \|P_{l,0}f\|
_{Y_q(\Omega_l)} + \sum \limits _{l\ge t} \sum \limits _{m\ge 0}
\|P_{l,m+1}f - P_{l,m}f\| _{Y_q(\Omega_l)}
\underset{\mathfrak{Z}_0}{\lesssim}
2^{-\frac{\alpha_*(s_*+1/q-1/p_1)-\mu_*(1/p_0-1/q)}
{s_*+1/p_0-1/p_1}k_*t}.
\end{align}
This estimate yields that
$$
f|_{\Omega_t} = \sum \limits _{l\ge t} P_{l,0}f + \sum \limits
_{l\ge t} \sum \limits _{m\ge 0} (P_{l,m+1}f - P_{l,m}f)
$$
(the series converges in $Y_q(\Omega)$) and (\ref{p0lq_emb})
holds.

First we consider the case $p_1\le q$. Then for $f\in
BX_{p_0}(\Omega)\cap BX_{p_1}(\Omega)$
$$
\sum \limits _{l\ge t} \|P_{l,0}f\| _{Y_q(\Omega_l)} + \sum
\limits _{l\ge t} \sum \limits _{m\ge 0} \|P_{l,m+1}f - P_{l,m}f\|
_{Y_q(\Omega_l)} \stackrel{(\ref{ptm1f_m_ptmf}),
(\ref{ptm_p011111}), (\ref{ptm1_m_ptm_p0})}
{\underset{\mathfrak{Z}_0}{\lesssim}}
$$
$$
\lesssim \sum \limits _{l\ge t} 2^{-\alpha_*k_*l} + \sum \limits
_{l\ge t} \sum \limits _{0\le m\le m_l} 2^{-\alpha_*k_*l}\cdot
2^{m\left(\frac{1}{p_0} -\frac 1q\right)} + \sum \limits _{l\ge t}
\sum \limits _{m> m_l} 2^{\mu_*k_*l}\cdot 2^{-m\left(s_* +\frac 1q
-\frac{1}{p_1}\right)}\underset{\mathfrak{Z}_0}{\lesssim}
$$
$$
\lesssim \sum \limits _{l\ge t} 2^{-\alpha_*k_*l}\cdot
2^{m_l\left(\frac{1}{p_0} -\frac 1q\right)} + \sum \limits _{l\ge
t} 2^{\mu_*k_*l}\cdot 2^{-m_l\left(s_* +\frac 1q
-\frac{1}{p_1}\right)} \stackrel{(\ref{mt_def})}
{\underset{\mathfrak{Z}}{\lesssim}}
$$
$$
\lesssim\sum \limits _{l\ge t} 2^{-\alpha_*k_*l}\cdot
2^{m_l\left(\frac{1}{p_0} -\frac 1q\right)}
\stackrel{(\ref{mt_form})} {\underset{\mathfrak{Z}}{\lesssim}}
2^{-\frac{\alpha_*(s_*+1/q-1/p_1)-\mu_*(1/p_0-1/q)}
{s_*+1/p_0-1/p_1}k_*t}.
$$

Let $p_1>q>p_0$. Then there exists $\lambda\in (0, \, 1)$ such
that $\frac 1q =\frac{1-\lambda}{p_1} +\frac{\lambda}{p_0}$.
Hence,
$$
\sum \limits _{l\ge t} \|P_{l,0}f\| _{Y_q(\Omega_l)} + \sum
\limits _{l\ge t} \sum \limits _{m\ge 0} \|P_{l,m+1}f - P_{l,m}f\|
_{Y_q(\Omega_l)} \stackrel{(\ref{ptm_p011111}),
(\ref{ptm1_m_ptm_p0}),
(\ref{wkm_lqd0})}{\underset{\mathfrak{Z}_0}{\lesssim}}
$$
$$
\lesssim \sum \limits _{l\ge t} 2^{-\alpha_*k_*l} + \sum \limits
_{l\ge t} \sum \limits _{0\le m\le m_l} 2^{-\alpha_*k_*l}\cdot
2^{m\left(\frac{1}{p_0} -\frac 1q\right)} + \sum \limits _{l\ge t}
\sum \limits _{m> m_l} 2^{((1-\lambda)\mu_*-\lambda
\alpha_*)k_*l}\cdot 2^{-ms_*(1-\lambda)}
\underset{\mathfrak{Z}_0}{\lesssim}
$$
$$
\lesssim \sum \limits _{l\ge t} 2^{-\alpha_*k_*l}\cdot
2^{m_l\left(\frac{1}{p_0} -\frac 1q\right)} + \sum \limits _{l\ge
t} 2^{((1-\lambda)\mu_*-\lambda \alpha_*)k_*l}\cdot 2^{-m_l
s_*(1-\lambda)}
\stackrel{(\ref{1lams}),(\ref{mt_def})}{\underset{\mathfrak{Z}_0}{\lesssim}}
$$
$$
\lesssim \sum \limits _{l\ge t} 2^{-\alpha_*k_*l}\cdot
2^{m_l\left(\frac{1}{p_0} -\frac 1q\right)}
\stackrel{(\ref{mt_form})} {\underset{\mathfrak{Z}}{\lesssim}}
2^{-\frac{\alpha_*(s_*+1/q-1/p_1)-\mu_*(1/p_0-1/q)}
{s_*+1/p_0-1/p_1}k_*t}.
$$

This completes the proof of (\ref{sum_emb}).
\end{proof}

Now we obtain the upper estimates for $d_n(BX_{p_0}(\Omega) \cap
BX_{p_1}(\Omega), \, Y_q(\Omega))$.

We need the following corollary from Gluskin's theorem
\cite{bib_gluskin}.
\begin{trma}
\label{gl_teor} Let $1\le p<q<\infty$, $q>2$, $\lambda _{pq} =\min
\left\{1, \, \frac{\frac 1p-\frac 1q}{\frac 12-\frac 1q}\right\}$,
$n\le N/2$. Then
$$
d_n(B_p^N, \, l_q^N) \underset{p,q}{\asymp} \min \{1, \,
n^{-1/2}N^{1/q}\} ^{\lambda_{pq}}.
$$
If $1\le p\le q\le 2$, $n\le N/2$, then
$$
d_n(B_p^N, \, l_q^N) \underset{p,q}{\asymp} 1.
$$
The upper estimates for the Kolmogorov widths also hold for $N/2<
n\le N$.
\end{trma}

For $1\le q\le p\le \infty$ the following equation holds
\cite{pietsch1}, \cite{stesin}:
\begin{align}
\label{pietsch_stesin} d_n(B_p^N, \, l_q^N) = (N-n)^{\frac
1q-\frac 1p}.
\end{align}

We denote
\begin{align}
\label{til_theta} \tilde \theta = \frac{s_*\left(\alpha_*
+\frac{\gamma_*}{p_0} -\frac{\gamma_*}{q}\right)}{\mu_* + \alpha_*
+\gamma_*\left(s_* +\frac{1}{p_0} -\frac{1}{p_1}\right)}, \quad
\hat \theta = \frac{\alpha_*\left(s_* +\frac 1q
-\frac{1}{p_1}\right) +\mu_*\left(\frac 1q
-\frac{1}{p_0}\right)}{\mu_* + \alpha_* +\gamma_*\left(s_*
+\frac{1}{p_0} -\frac{1}{p_1}\right)}.
\end{align}
Let the numbers $j_0\in \N$ and $\theta_j\in \R$ ($1\le j\le j_0$)
be as in Definition \ref{theta_j}.

\begin{Trm}
Let
\begin{align}
\label{s1qp} \min \left\{s_*, \, s_* +\frac 1q -\frac{1}{p_1},\;
s_*+\frac{1}{p_0} -\frac{1}{p_1}\right\}>0,
\end{align}
\begin{align}
\label{mua} \min\left\{\mu_* +\alpha_*+\frac{\gamma_*}{p_0}
-\frac{\gamma_*}{p_1}, \, \mu_*+\alpha_* \right\}>0.
\end{align}
Suppose that $\alpha_*> \frac{\gamma_*}{q} -\frac{\gamma_*}{p_0}$
for $p_0 \ge q$, $\alpha_*\left(s_*+\frac 1q-\frac{1}{p_1}\right)>
\mu_*\left(\frac{1}{p_0}-\frac 1q\right)$ for $p_0\le q$, $p_1\le
q$ or $p_0<q$, $p_1>q$. Suppose that there is $j_*\in \{1, \,
\dots, \, j_0\}$ such that $\theta_{j_*}< \min _{j\ne j_*}
\theta_j$. Then
$$
d_n(BX_{p_0}(\Omega) \cap BX_{p_1}(\Omega), \, Y_q(\Omega))
\underset{\mathfrak{Z}_0}{\lesssim} n^{-\theta_{j_*}}.
$$
\end{Trm}
\begin{proof}
We define the numbers $\hat m_t$, $\overline{m}_t$, $\tilde m_t$,
$m_t\in \R$ by equations
\begin{align}
\label{hat_mt} 2^{\gamma_*k_*t}\cdot 2^{\hat m_t}=n,
\end{align}
\begin{align}
\label{line_mt} 2^{\gamma_*k_*t}\cdot 2^{\overline{m}_t}=n^{q/2}
\quad(\text{for }q>2),
\end{align}
\begin{align}
\label{til_mt} 2^{-(\alpha_*+\gamma_*/p_0 - \gamma_*/q)k_*t} =
2^{(\mu_*+\gamma_*/q-\gamma_*/p_1)k_*t}\cdot 2^{-s_*\tilde m_t},
\end{align}
\begin{align}
\label{mt} 2^{-\alpha_*k_*t}\cdot 2^{m_t(1/p_0-1/q)} =
2^{\mu_*k_*t}\cdot 2^{-m_t(s_*+1/q-1/p_1)};
\end{align}
$\tilde t(n)$ and $t(n)$ are defined by equations
\begin{align}
\label{1111} \hat m_{\tilde t(n)}=\tilde m_{\tilde t(n)}, \quad
\hat m_{t(n)}=m_{t(n)}.
\end{align}

Then
\begin{align}
\label{til_mt_t} 2^{\tilde m_t s_*} = 2^{(\mu_* + \alpha_*
+\gamma_*/p_0 -\gamma_*/p_1)k_*t},
\end{align}
\begin{align}
\label{mt_t} 2^{m_t(s_*+1/p_0-1/p_1)} = 2^{(\mu_*+\alpha_*) k_*t},
\end{align}
\begin{align}
\label{tn} 2^{(\mu_*+\alpha_*+\gamma_*(s_*+1/p_0-1/p_1))k_*\tilde
t(n)} = n^{s_*},
\end{align}
\begin{align}
\label{tn1} 2^{(\mu_*+\alpha_*+\gamma_*(s_*+1/p_0-1/p_1))k_* t(n)}
= n^{s_*+1/p_0-1/p_1}.
\end{align}
If $q>2$, we also define the number $\hat t(n)$ by equation
\begin{align}
\label{2222} \overline{m}_{\hat t(n)} = m_{\hat t(n)}.
\end{align}
Then
\begin{align}
\label{tn_hat} 2^{(\mu_*+\alpha_*+\gamma_*(s_*+1/p_0-1/p_1))k_*
\hat t(n)} = n^{(s_*+1/p_0-1/p_1)q/2}.
\end{align}

Notice that from (\ref{s1qp}), (\ref{mua}), (\ref{1111}),
(\ref{til_mt_t}), (\ref{mt_t}), (\ref{tn}), (\ref{tn1}),
(\ref{2222}), (\ref{tn_hat}) it follows that
\begin{align}
\label{2t} 2^{\hat m_{t(n)}} = n^{\beta_1}, \quad 2^{\hat
m_{\tilde t(n)}}=n^{\beta_2}, \quad 2^{\overline{m}_{t(n)}} =
n^{\beta_3}, \quad \beta_i>0, \quad i=1, \, 2, \, 3.
\end{align}

Further $\varepsilon>0$ is a sufficiently small number; it will be
chosen later by $\mathfrak{Z}_0$.

First we consider $p_0\ge q$.

{\bf Case $p_0\ge q$, $p_1\ge q$.} We define the numbers ${m}^*_t$
by
\begin{align}
\label{2mt} 2^{{m}^*_t} = 2^{\hat m_t-\varepsilon|t-t_*(n)|};
\end{align}
here $t_*(n)=0$ or $t_*(n)=\tilde t(n)$ (we will choose $t_*(n)$
later by $\mathfrak{Z}_0$). By (\ref{hat_mt}), (\ref{tn}) and
(\ref{2t}), if $\varepsilon$ is sufficiently small, we have
$m_t^*>0$ for $0\le t\le \tilde t(n)$.

We set $Pf|_{\Omega_t} = P_{t,[{m}^*_t]}f$, $t_0\le t\le [\tilde
t(n)]$, $Pf|_{\tilde \Omega_{[\tilde t(n)]+1}} =0$. Then $${\rm
rk}\, P \stackrel{(\ref{rk_ptm}), (\ref{hat_mt}), (\ref{2mt})}
{\underset{\mathfrak{Z}_0, \varepsilon}{\lesssim}} n,$$
$$
\|f-Pf\| _{Y_q(\Omega)} \le \sum \limits _{t_0\le t\le [\tilde
t(n)]}\|f-P_{t,m}f\|_{Y_q(\Omega_t)} + \|f\|_{Y_q(\tilde
\Omega_{[\tilde t(n)]+1})} \stackrel{(\ref{f_m_ptm}),
(\ref{p0geq_emb}), (\ref{hat_mt}), (\ref{2mt})}
{\underset{\mathfrak{Z}_0}{\lesssim}}
$$
$$
\lesssim \sum \limits _{0\le t\le \tilde t(n)}
2^{(\mu_*+\gamma_*/q-\gamma_*/p_1)k_*t} \cdot
2^{s_*\gamma_*k_*t}\cdot n^{-s_*}\cdot
2^{s_*\varepsilon|t-t_*(n)|} +
2^{-(\alpha_*+\gamma_*/p_0-\gamma_*/q)k_*\tilde t(n)}
\stackrel{(\ref{tn})}{\underset{\mathfrak{Z}_0}{\lesssim}}
$$
$$
\lesssim \sum \limits _{0\le t\le \tilde t(n)}
2^{(\mu_*+\gamma_*(s_*+1/q-1/p_1))k_*t} \cdot n^{-s_*}\cdot
2^{s_*\varepsilon|t-t_*(n)|}+
2^{(\mu_*+\gamma_*(s_*+1/q-1/p_1))k_*\tilde t(n)} \cdot
n^{-s_*}=:S.
$$
If $\mu_*+\gamma_*(s_*+1/q-1/p_1)<0$, then we set $t_*(n)=0$ and
get $S \underset{\mathfrak{Z}_0}{\lesssim} n^{-s_*}$. If
$\mu_*+\gamma_*(s_*+1/q-1/p_1)>0$, then we set $t_*(n)=\tilde
t(n)$ and obtain
$$
S \underset{\mathfrak{Z}_0}{\lesssim}
2^{(\mu_*+\gamma_*(s_*+1/q-1/p_1))k_*\tilde t(n)} \cdot n^{-s_*}
\stackrel{(\ref{til_theta}),
(\ref{tn})}{\underset{\mathfrak{Z}_0}{\asymp}} n^{-\tilde \theta}.
$$
By conditions of theorem, $s_*\ne \tilde \theta$; therefore,
$\mu_*+\gamma_*(s_*+1/q-1/p_1) \ne 0$.

{\bf Case $p_0>q$, $p_1< q\le 2$.} We define $\lambda \in (0, \,
1)$ by equation
\begin{align}
\label{1q1l1p1} \frac 1q = \frac{1-\lambda}{p_1} +
\frac{\lambda}{p_0}.
\end{align}
Then (\ref{wkm_lqd0}) holds.

Since $p_1<q<p_0$, we have $t(n)<\tilde t(n)$ by (\ref{tn}),
(\ref{tn1}). Let $t_*(n)=0$ or $t_*(n)=t(n)$, $t_{**}(n)=t(n)$ or
$t_{**}(n) = \tilde t(n)$ (they will be chosen later by
$\mathfrak{Z}_0$). We define the numbers $m_t^*$ ($t_0\le t\le
t(n)$) and $m_t^{**}$ ($t(n)< t\le \tilde t(n)$) by
\begin{align}
\label{2_mt_st} 2^{m_t^*}=2^{\hat m_t-\varepsilon|t-t_*(n)|},
\quad 2^{m_t^{**}} = 2^{\hat m_t-\varepsilon|t-t_{**}(n)|}.
\end{align}
By (\ref{hat_mt}), (\ref{tn}) and (\ref{2t}), for small
$\varepsilon>0$ and $t_0\le t\le \tilde t(n)$ the numbers $m_t^*$
and $m_t^{**}$ are positive.

For $f\in BX_{p_0}(\Omega)\cap BX_{p_1}(\Omega)$ the following
equation holds:
$$
f = f\cdot \chi _{\tilde \Omega _{[\tilde t(n)]+1}}+\sum \limits
_{t_0\le t\le t(n)} P_{t,[m_t^*]}f + \sum \limits _{t(n)<t\le
[\tilde t(n)]} P_{t,[m_t^{**}]}f+
$$
$$
+ \sum \limits _{t_0\le t\le t(n)} (f - P_{t,[m_t^*]}f) + \sum
\limits _{t(n)<t\le [\tilde t(n)]} \sum \limits _{m\ge [m_t^{**}]}
(P_{t,m+1}f-P_{t,m}f).
$$
We have
$$
\sum \limits _{t_0\le t\le t(n)} {\rm rk}\, P_{t,[m_t^*]} + \sum
\limits _{t(n)<t\le \tilde t_n} {\rm rk}\, P_{t,[m_t^{**}]}
\stackrel{(\ref{rk_ptm}), (\ref{2_mt_st})}{\underset
{\mathfrak{Z}_0}{\lesssim}}
$$
$$
\lesssim \sum \limits _{t_0\le t\le t(n)} 2^{\gamma_*k_*t}\cdot
2^{\hat m_t}\cdot 2^{-\varepsilon|t-t_*(n)|} + \sum \limits
_{t(n)<t\le \tilde t(n)} 2^{\gamma_*k_*t}\cdot 2^{\hat m_t}\cdot
2^{-\varepsilon|t-t_{**}(n)|}
\stackrel{(\ref{hat_mt})}{\underset{\varepsilon,
\mathfrak{Z}_0}{\lesssim}} n,
$$
$$
\|f\|_{Y_q(\tilde \Omega _{[\tilde t(n)]+1})}
\stackrel{(\ref{p0geq_emb})}{\underset{\mathfrak{Z}_0}{\lesssim}}
2^{-(\alpha_*+\gamma_*/p_0-\gamma_*/q) k_*\tilde t(n)}
\stackrel{(\ref{til_theta}),
(\ref{tn})}{\underset{\mathfrak{Z}_0}{\lesssim}} n^{-\tilde
\theta}.
$$
This together with (\ref{f_m_ptm}), (\ref{wkm_lqd0}) yield that it
remains to estimate the sum
$$
\sum \limits _{t_0\le t\le t(n)} 2^{\mu_*k_*t}\cdot
2^{-m_t^*(s_*+1/q -1/p_1)} + \sum \limits _{t(n)<t \le \tilde
t(n)} \sum \limits _{m\ge m_t^{**}} 2^{((1-\lambda)\mu_*-\lambda
\alpha_*)k_*t}\cdot 2^{-ms_*(1-\lambda)}
\stackrel{(\ref{2_mt_st})}{\underset{\mathfrak{Z}_0}{\lesssim}}
$$
$$
\lesssim \sum \limits _{0\le t\le t(n)} 2^{\mu_*k_*t}\cdot
2^{-\hat m_t(s_*+1/q -1/p_1)+\varepsilon (s_*+1/q
-1/p_1)|t-t_*(n)|} +
$$
$$
+\sum \limits _{t(n)<t \le \tilde t(n)}
2^{((1-\lambda)\mu_*-\lambda \alpha_*)k_*t}\cdot 2^{-\hat
m_ts_*(1-\lambda)+\varepsilon (1-\lambda)s_* |t-t_{**}(n)|} =:S.
$$

By (\ref{1lams}), (\ref{1111}) and (\ref{mt_t}), we get
$$
2^{\mu_*k_*t(n)}\cdot 2^{-\hat m_{t(n)}(s_*+1/q -1/p_1)}
\underset{\mathfrak{Z}_0}{\asymp} 2^{((1-\lambda)\mu_*-\lambda
\alpha_*)k_*t(n)}\cdot 2^{-\hat m_{t(n)}s_*(1-\lambda)}.
$$

Recall that $\theta_{j_*}<\min _{j\ne j_*} \theta_j$ by conditions
of theorem. Taking into account (\ref{hat_mt}) and appropriately
choosing $t_*(n)$ and $t_{**}(n)$ we get $S
\underset{\mathfrak{Z}_0}{\lesssim} S_1(n)+ S_2(n)+S_3(n)$, where
$$
S_1(n)= 2^{-\hat m_0(s_*+1/q-1/p_1)} \stackrel{(\ref{hat_mt})}{=}
n^{-s_*-1/q+1/p_1},
$$
$$
S_2(n) = 2^{\mu_*k_*t(n)}\cdot 2^{-\hat m_{t(n)}(s_*+1/q -1/p_1)}
\stackrel{(\ref{til_theta}), (\ref{hat_mt}),
(\ref{tn1})}{\underset{\mathfrak{Z}_0}{\asymp}} n^{-\hat\theta},
$$
$$
S_3(n) = 2^{((1-\lambda)\mu_*-\lambda \alpha_*)k_*\tilde
t(n)}\cdot 2^{-\hat m_{\tilde t(n)}s_*(1-\lambda)} \stackrel{
(\ref{1111}), (\ref{til_mt_t})}{\underset{\mathfrak{Z}_0}{\asymp}}
$$
$$
\asymp
2^{-(\alpha_*+(1-\lambda)(\gamma_*/p_0-\gamma_*/p_1))k_*\tilde
t(n)} \stackrel{(\ref{til_theta}), (\ref{tn}),
(\ref{1q1l1p1})}{\underset{\mathfrak{Z}_0}{\asymp}} n^{-\tilde
\theta}.
$$

{\bf Case $p_0>q>2$, $q>p_1\ge 2$.} The numbers $t_1(n)\le \tilde
t(n)$, $m_1(n)\le n^{q/2}$ will be chosen later by
$\mathfrak{Z}_0$. For $0\le t\le \tilde t(n)$ we define the
numbers $m_t^*$ by
\begin{align}
\label{mt_st} 2^{m_t^*} = 2^{\hat m_t-\varepsilon |t-t_1(n)|}.
\end{align}
As in previous cases, we get $m_t^*>0$ for small $\varepsilon$.
For $0\le t\le \tilde t(n)$ we set
\begin{align}
\label{l_tm} l_{t,m} = \left\{ \begin{array}{l} \left \lceil
n\cdot 2^{-\varepsilon(|t-t_1(n)|+|m-m_1(n)|)}\right \rceil, \quad
m\le \overline{m}_t, \\ 0, \quad m>\overline{m}_t.
\end{array}\right.
\end{align}
Then for sufficiently small $\varepsilon>0$
\begin{align}
\label{pmt_rk} \sum \limits _{t=t_0}^{[\tilde t(n)]} {\rm rk}\,
P_{t,\lceil m_t^*\rceil} \stackrel{(\ref{rk_ptm}),(\ref{hat_mt}),
(\ref{mt_st})}{\underset{\mathfrak{Z}_0,\varepsilon}{\lesssim}} n,
\quad \sum \limits _{t=t_0}^{[\tilde t(n)]} \sum \limits _{m\ge
\lceil m_t^*\rceil } l_{t,m}
\underset{\mathfrak{Z}_0,\varepsilon}{\lesssim} n.
\end{align}

For $f\in BX_{p_0}(\Omega)\cap BX_{p_1}(\Omega)$ we have
\begin{align}
\label{f_f_cdot} f = f\cdot \chi_{\tilde \Omega _{[\tilde
t(n)]+1}} + \sum \limits _{t=t_0}^{[\tilde t(n)]} P_{t,\lceil
m_t^*\rceil }f + \sum \limits _{t=t_0}^{[\tilde t(n)]} \sum
\limits _{m\ge [m_t^*]} (P_{t,m+1} f - P_{t,m}f).
\end{align}

By (\ref{p0geq_emb}), (\ref{til_theta}) and (\ref{tn}), $\|f\cdot
\chi_{\tilde \Omega _{[\tilde t(n)]+1}}\|_{Y_q(\Omega)}
\underset{\mathfrak{Z}_0}{\lesssim} n^{-\tilde \theta}$. This
together with (\ref{dl}), (\ref{pmt_rk}) and (\ref{f_f_cdot})
yields that it remains to estimate the sum
\begin{align}
\label{sum_dl} \sum \limits _{t=0}^{[\tilde t(n)]} \sum \limits
_{m\ge [m_t^*]} d_{l_{t,m}} (W_{t,m}, \, l_q^{\nu_{t,m}}) =: S.
\end{align}
Let
\begin{align}
\label{1q_lam} \frac 1q =\frac{1-\lambda}{p_1}
+\frac{\lambda}{p_0}.
\end{align}
By Theorem \ref{gl_teor} and (\ref{rk_ptm}), (\ref{w_tm}),
(\ref{wkm_lq}), we have
\begin{align}
\label{dltmp} d_{l_{t,m}}(W_{t,m}, \, l_q^{\nu_{t,m}})
\underset{\mathfrak{Z}_0} {\lesssim} 2^{\mu_*k_*t}\cdot
2^{-(s_*+1/q-1/p_1)m}\left(l_{t,m}^{-1/2}\cdot
2^{\gamma_*k_*t/q}\cdot
2^{m/q}\right)^{\frac{1/p_1-1/q}{1/2-1/q}},
\end{align}
\begin{align}
\label{dltmq} d_{l_{t,m}}(W_{t,m}, \, l_q^{\nu_{t,m}})
\underset{\mathfrak{Z}_0} {\lesssim}
2^{((1-\lambda)\mu_*-\lambda\alpha_*)k_*t}\cdot
2^{-(1-\lambda)s_*m}.
\end{align}
If $\hat t(n)<\tilde t(n)$, we define the numbers $m_t'$ by
$$
2^{\mu_*k_*t}\cdot 2^{-(s_*+1/q-1/p_1)m_t'}\left(n^{-1/2}\cdot
2^{\gamma_*k_*t/q}\cdot
2^{m'_t/q}\right)^{\frac{1/p_1-1/q}{1/2-1/q}} =
2^{((1-\lambda)\mu_*-\lambda\alpha_*)k_*t}\cdot
2^{-(1-\lambda)s_*m_t'}.
$$
Then by (\ref{1lams}) and (\ref{1q_lam}) we get
\begin{align}
\label{3333} 2^{(\mu_*+\alpha_*)k_*t}\cdot
2^{-(s_*+1/p_0-1/p_1)m_t'}\left( n^{-1/2}\cdot
2^{\gamma_*k_*t/q}\cdot 2^{m_t'/q}\right)
^{\frac{1/p_1-1/p_0}{1/2-1/q}} =1.
\end{align}
If $m_t'=\hat m_t$, then $t=\tilde t(n)$ by (\ref{hat_mt}) and
(\ref{tn}); if $m_t'=\overline{m}_t$, then $t=\hat t(n)$ by
(\ref{line_mt}) and (\ref{tn_hat}). Notice that if the factor
multiplying  $m_t'$ in the exponent in left-hand-side of
(\ref{3333}) is zero, then $\tilde t(n)=\hat t(n)$.

We split the set $\{(t, \, m):\; 0\le t\le \tilde t(n), \;
m_t^*\le m<\infty\}$ into the following subsets:
$$
{\rm I} = \{(t, \, m):\; 0\le t\le \tilde t(n), \; m_t^*\le m\le
\overline{m}_t; \;m\le m_t', \; \text{if }\hat t(n)<t\le \tilde
t(n)\},
$$
$$
{\rm II}=\{(t, \, m):\; 0\le t\le \min(\tilde t(n), \, \hat t(n)),
\; m\ge \overline{m}_t\},
$$
$$
{\rm III} = \{(t, \, m):\; \hat t(n)<t\le \tilde t(n),\; m\ge
m_t'\}.
$$

For $(t, \, m)\in {\rm I} \cup {\rm II}$ we apply (\ref{dltmp}),
for $(t, \, m)\in {\rm III}$ we apply (\ref{dltmq}).

We apply Lemma 6 from \cite{vas_bes}, (\ref{hat_mt}), take into
account that $\theta _{j_*}<\min _{j\ne j_*} \theta_j$ by theorem
conditions, chose appropriately the numbers $\varepsilon$,
$t_1(n)$, $m_1(n)$ and obtain the estimate $S
\underset{\mathfrak{Z}_0}{\lesssim} S_1(n)+S_2(n)+S_3(n)+S_4(n)$,
where
$$
S_1(n) =2^{-(s_*+1/q-1/p_1)\hat m_{0}}\left(n^{-1/2}\cdot
 2^{\hat
m_{0}/q}\right)^{\frac{1/p_1-1/q}{1/2-1/q}}
\stackrel{(\ref{hat_mt})}{\underset{\mathfrak{Z}_0}{\lesssim}}
n^{-s_*},
$$
$$
S_2(n) =2^{-(s_*+1/q-1/p_1)\overline{m} _{0}}\left(n^{-1/2}\cdot
2^{\overline{m}_{0}/q}\right)^{\frac{1/p_1-1/q}{1/2-1/q}}
\stackrel{(\ref{line_mt})}{\underset{\mathfrak{Z}_0}{\lesssim}}
n^{-q(s_*+1/q-1/p_1)/2},
$$
$$
S_3(n) =2^{\mu_*k_*\tilde t(n)}\cdot
2^{-(s_*+1/q-1/p_1)\hat{m}_{\tilde t(n)}}\left(n^{-1/2}\cdot
2^{\gamma_*k_*\tilde t(n)/q}\cdot 2^{\hat{m}_{\tilde
t(n)}/q}\right)^{\frac{1/p_1-1/q}{1/2-1/q}}
\stackrel{(\ref{til_theta}),(\ref{hat_mt}),
(\ref{tn})}{\underset{\mathfrak{Z}_0}{\lesssim}} n^{-\tilde
\theta},
$$
$$
S_4(n)=2^{\mu_*k_*\hat t(n)}\cdot
2^{-(s_*+1/q-1/p_1)\overline{m}_{\hat t(n)}}\left(n^{-1/2}\cdot
2^{\gamma_*k_*\hat t(n)/q}\cdot 2^{\overline{m}_{\hat
t(n)}/q}\right)^{\frac{1/p_1-1/q}{1/2-1/q}}
\stackrel{(\ref{til_theta}), (\ref{line_mt}),
(\ref{tn_hat})}{\underset{\mathfrak{Z}_0}{\lesssim}} n^{-q\hat
\theta/2}.
$$

{\bf Case $p_0>q>2>p_1$.} Let $\frac 1q =\frac{1-\lambda}{p_1}
+\frac{\lambda}{p_0}$, $\frac 12 =\frac{1-\tilde \lambda}{p_1} +
\frac{\tilde \lambda}{p_0}$. As in previous case, we define
$m_t^*$ and $l_{t,m}$ by (\ref{mt_st}) and (\ref{l_tm}) and get
that it suffices to estimate the sum (\ref{sum_dl}). By Theorem
\ref{gl_teor}, (\ref{rk_ptm}), (\ref{w_tm}), (\ref{wkm_lq}) and
(\ref{wkm_l2}), we get
\begin{align}
\label{dltmp1} d_{l_{t,m}}(W_{t,m}, \, l_q^{\nu_{t,m}})
\underset{\mathfrak{Z}_0} {\lesssim} 2^{\mu_*k_*t}\cdot
2^{-(s_*+1/q-1/p_1)m}l_{t,m}^{-1/2}\cdot 2^{\gamma_*k_*t/q}\cdot
2^{m/q},
\end{align}
\begin{align}
\label{dltmq1} d_{l_{t,m}}(W_{t,m}, \, l_q^{\nu_{t,m}})
\underset{\mathfrak{Z}_0} {\lesssim}
2^{((1-\lambda)\mu_*-\lambda\alpha_*)k_*t}\cdot
2^{-(1-\lambda)s_*m},
\end{align}
\begin{align}
\label{dltm2} \begin{array}{c} d_{l_{t,m}}(W_{t,m}, \,
l_q^{\nu_{t,m}}) \underset{\mathfrak{Z}_0} {\lesssim} \\ \lesssim
2^{((1-\tilde \lambda)\mu_*-\tilde \lambda \alpha_*)k_*t}\cdot
2^{-((1-\tilde\lambda)(s_*+1/q-1/p_1)+\tilde \lambda(1/q-1/p_0))m}
l_{t,m}^{-1/2}\cdot 2^{\gamma_*k_*t/q}\cdot 2^{m/q}. \end{array}
\end{align}
We define the numbers $m_t'$ by
$$
2^{((1-\lambda)\mu_*-\lambda\alpha_*)k_*t}\cdot
2^{-(1-\lambda)s_*m_t'} =
$$
$$
=2^{((1-\tilde \lambda)\mu_*-\tilde \lambda \alpha_*)k_*t}\cdot
2^{-((1-\tilde\lambda)(s_*+1/q-1/p_1)+\tilde
\lambda(1/q-1/p_0))m_t'} n^{-1/2}\cdot 2^{\gamma_*k_*t/q}\cdot
2^{m_t'/q}.
$$
Taking into account (\ref{1lams}), we get that
$$
2^{(\mu_*+\alpha_*)k_*t}\cdot 2^{-(s_*+1/p_0-1/p_1)m'_t}\cdot
\left(n^{-1/2}\cdot 2^{\gamma_*k_*t}\cdot
2^{m'_t/q}\right)^{\frac{1/p_1-1/p_0}{1/2-1/q}}=1.
$$
Notice that if the factor multiplying of $m_t'$ in the exponent is
zero, then $\hat t(n) = \tilde t(n)$.

If $m'_t=\hat m_t$, then $t=\tilde t(n)$ by (\ref{hat_mt}) and
(\ref{tn}); if $m'_t=\overline{m}_t$, then $t=\hat t(n)$ and
$m'_t=m_t$ by (\ref{line_mt}) and (\ref{tn_hat}).

We split the set $\{(t, \, m):\; 0\le t\le \tilde t(n), \;
m_t^*\le m<\infty\}$ into the subsets
$$
{\rm I} = \{(t, \, m):\; 0\le t\le \tilde t(n), \; m_t^*\le
m\le\overline{m}_t, \; m\ge m_t\},
$$
$$
{\rm II}=\{(t, \, m):\; 0\le t\le \min(\hat t(n), \, \tilde t(n)),
\; m\ge \overline{m}_t\},
$$
$$
{\rm III}= \{(t, \, m):\; 0\le t\le \tilde t(n), \; m\ge m_t^*, \;
m\le m_t; \; m\le m_t'\text{ for }\hat t(n)<t\le \tilde t(n)\},
$$
$$
{\rm IV}=\{(t, \, m):\; \hat t(n)<t\le \tilde t(n), \; m\ge
m_t'\}.
$$

For $(t, \, m)\in {\rm I}\cup {\rm II}$ we apply (\ref{dltmp1});
for $(t, \, m)\in {\rm III}$ we apply (\ref{dltm2}); for $(t, \,
m)\in {\rm IV}$ we apply (\ref{dltmq1}).

As in the previous case, we apply Lemma 6 from \cite{vas_bes} and
get that for appropriate $\varepsilon>0$, $t_1(n)$ and $m_1(n)$
the estimate $S\underset{\mathfrak{Z}_0}{\lesssim}
S_1(n)+S_2(n)+S_3(n)+S_4(n)+S_5(n)$ holds with
$$
S_1(n) =2^{-(s_*+1/q-1/p_1)\hat m_{0}}n^{-1/2}\cdot 2^{\hat
m_{0}/q} \stackrel{(\ref{hat_mt})} {\underset{\mathfrak{Z}_0}
{\lesssim}} n^{-s_*-\frac 12 +\frac{1}{p_1}},
$$
$$
S_2(n) = 2^{-(s_*+1/q-1/p_1)\overline{m}_{0}}n^{-1/2}\cdot
2^{\overline{m}_{0}/q}
\stackrel{(\ref{line_mt})}{\underset{\mathfrak{Z}_0}{\lesssim}}
n^{-q(s_*+1/q-1/p_1)/2},
$$
$$
S_3(n) =2^{\mu_*k_*t(n)}\cdot 2^{-(s_*+1/q-1/p_1) \hat
m_{t(n)}}n^{-1/2}\cdot 2^{\gamma_*k_*t(n)/q}\cdot 2^{\hat
m_{t(n)}/q} \stackrel{(\ref{til_theta}), (\ref{hat_mt}),
(\ref{tn1})}{\underset{\mathfrak{Z}_0}{\lesssim}} n^{-\hat \theta
-\frac 12 +\frac 1q},
$$
$$
S_4(n) = 2^{\mu_*k_*\hat t(n)}\cdot
2^{-(s_*+1/q-1/p_1)\overline{m}_{\hat t(n)}}n^{-1/2}\cdot
2^{\gamma_*k_*\hat t(n)/q}\cdot 2^{\overline{m}_{\hat t(n)}/q}
\stackrel{(\ref{til_theta}), (\ref{line_mt}),
(\ref{tn_hat})}{\underset{\mathfrak{Z}_0}{\lesssim}} n^{-q\hat
\theta/2},
$$
$$
S_5(n) =2^{((1-\lambda)\mu_*-\lambda\alpha_*)k_*\tilde t(n)}\cdot
2^{-(1-\lambda)s_*\hat m_{\tilde t(n)}} \stackrel{(\ref{1111}),
(\ref{til_mt_t})}{\underset{\mathfrak{Z}_0}{\asymp}}
$$
$$
\asymp 2^{(1-\lambda)(\mu_*+\alpha_*)k_*\tilde t(n)}\cdot
2^{-\alpha_*k_*\tilde t(n)}\cdot 2^{-(1-\lambda)
(\mu_*+\alpha_*+\gamma_*/p_0-\gamma_*/p_1)k_*\tilde
t(n)}\stackrel{(\ref{til_theta}),
(\ref{tn})}{\underset{\mathfrak{Z}_0}{\asymp}} n^{-\tilde \theta}.
$$

\smallskip

Now we consider $p_0\le q$.

{\bf Case $p_0\le q\le 2$, $p_1\le q$.} Let $t_*(n)=0$ or
$t_*(n)=t(n)$ (it will be chosen later by $\mathfrak{Z}_0$). We
define the numbers $m_t^*$ by equation $2^{m_t^*} = 2^{\hat m_t
-\varepsilon |t-t_*(n)|}$. By (\ref{tn1}) and (\ref{2t}),
$m_t^*>0$ for $t_0\le t\le t(n)$ and small $\varepsilon>0$.

We set $Pf|_{\Omega_t}=P_{t,\lceil m_t^*\rceil}f$, $t_0\le t\le
[t(n)]$, $Pf|_{\tilde\Omega _{[t(n)]+1}}=0$. Then ${\rm rk}\, P
\stackrel{(\ref{rk_ptm}),(\ref{hat_mt})}{\underset{\mathfrak{Z}_0}{\lesssim}}
n$,
$$
\|f-Pf\|_{Y_q(\Omega)} \stackrel{(\ref{f_m_ptm}),(\ref{p0lq_emb}),
(\ref{hat_mt})}{\underset{\mathfrak{Z}_0}{\lesssim}}
2^{-\frac{\alpha_*(s_*+1/q-1/p_1)
-\mu_*(1/p_0-1/q)}{s_*+1/p_0-1/p_1}k_*t(n)} + $$$$+\sum \limits
_{t_0\le t\le t(n)} 2^{\mu_*k_*t}\cdot
2^{(s_*+1/q-1/p_1)\gamma_*k_*t}\cdot n^{-s_*-1/q+1/p_1}\cdot
2^{\varepsilon(s_*+1/q-1/p_1)|t-t_*(n)|}.
$$
If $\mu_*+\gamma_*(s_*+1/q-1/p_1)<0$, we set $t_*(n)=0$ and for
small $\varepsilon>0$ we get $\|f-Pf\|_{Y_q(\Omega)}
\underset{\mathfrak{Z}_0}{\lesssim} n^{-s_*-1/q+1/p_1}$; if
$\mu_*+\gamma_*(s_*+1/q-1/p_1)>0$, we set $t_*(n)=t(n)$ and get
$\|f-Pf\|_{Y_q(\Omega)} \stackrel{(\ref{til_theta}), (\ref{tn1})}
{\underset{\mathfrak{Z}_0} {\lesssim}} n^{-\hat \theta}$. If
$\mu_*+\gamma_*(s_*+1/q-1/p_1)=0$, then $s_*+\frac
1q-\frac{1}{p_1}=\hat \theta$, which contradicts with theorem
conditions.

{\bf Case $p_0< q\le 2$, $p_1>q$}. Here we argue as for $p_0>q$,
$p_1<q\le 2$; the number $\lambda$ is defined by (\ref{1q1l1p1}).
The numbers $m_t^*$, $m_t^{**}$ are defined by (\ref{2_mt_st}).
Notice that by (\ref{tn}), (\ref{tn1}) and inequality $p_0<p_1$ we
have $\tilde t(n)<t(n)$. It remains to estimate the sum
$$
\sum \limits _{t_0\le t\le \tilde t(n)}
\|f-P_{t,m^*_t}f\|_{Y_q(\Omega_t)} + \sum \limits _{\tilde
t(n)<t\le t(n)} \sum \limits _{m\ge m_t^{**}}
\|P_{t,m+1}f-P_{t,m}f\|_{Y_q(\Omega_t)}
+\|f\|_{Y_q(\Omega_{[t(n)]+1})}=:A.
$$
By (\ref{f_m_ptm}), (\ref{wkm_lqd0}), (\ref{p0lq_emb}),
(\ref{til_theta}), (\ref{hat_mt}), (\ref{tn1}), (\ref{2_mt_st})
$$
A\underset{\mathfrak{Z}_0}{\lesssim}\sum \limits _{0\le t\le
\tilde t(n)} 2^{(\mu_*+\gamma_*/q-\gamma_*/p_1)k_*t}\cdot
2^{s_*\gamma_*k_*t}\cdot n^{-s_*}\cdot
2^{s_*\varepsilon|t-t_*(n)|} +
$$
$$
+\sum \limits _{\tilde t(n)<t\le t(n)} \sum \limits _{m\ge
m_t^{**}} 2^{((1-\lambda)\mu_*-\lambda\alpha_*)k_*t}\cdot
2^{-ms_*(1-\lambda)} + n^{-\hat{\theta}}
\stackrel{(\ref{2_mt_st})}{\underset{\mathfrak{Z}_0}{\lesssim}}
$$
$$
\lesssim \sum \limits _{0\le t\le \tilde t(n)}
2^{(\mu_*+\gamma_*/q-\gamma_*/p_1)k_*t}\cdot
2^{s_*\gamma_*k_*t}\cdot n^{-s_*}\cdot
2^{s_*\varepsilon|t-t_*(n)|} +
$$
$$
+\sum \limits _{\tilde t(n)<t\le t(n)}
2^{((1-\lambda)\mu_*-\lambda\alpha_*)k_*t}\cdot
2^{\gamma_*s_*(1-\lambda)k_*t}\cdot n^{-s_*(1-\lambda)} \cdot
2^{\varepsilon s_*|t-t_{**}(n)|(1-\lambda)} +
n^{-\hat{\theta}}=:S.
$$
Since $\theta_{j_*}<\min _{j\ne j_*}\theta_j$ by theorem
conditions and
$$
2^{(\mu_*+\gamma_*/q-\gamma_*/p_1)k_*\tilde t(n)}\cdot
2^{s_*\gamma_*k_*\tilde t(n)}\cdot n^{-s_*}
\stackrel{(\ref{tn})}{=}
2^{((1-\lambda)\mu_*-\lambda\alpha_*)k_*\tilde t(n)}\cdot
2^{\gamma_*s_*(1-\lambda)k_*\tilde t(n)}\cdot n^{-s_*(1-\lambda)},
$$
we have $S \underset{\mathfrak{Z}_0}{\lesssim}
S_1(n)+S_2(n)+S_3(n)+n^{-\hat\theta}$, where
$$
S_1(n)= n^{-s_*}, \quad S_2(n)
=2^{(\mu_*+\gamma_*/q-\gamma_*/p_1)k_*\tilde t(n)}\cdot
2^{s_*\gamma_*k_*\tilde t(n)}\cdot n^{-s_*}
\stackrel{(\ref{til_theta}),
(\ref{tn})}{\underset{\mathfrak{Z}_0}{\lesssim}} n^{-\tilde
\theta},
$$
$$
S_3(n) = 2^{((1-\lambda)\mu_*-\lambda\alpha_*)k_*t(n)}\cdot
2^{\gamma_*s_*(1-\lambda)k_*t(n)}\cdot n^{-s_*(1-\lambda)}
\stackrel{(\ref{til_theta}), (\ref{tn1}),
(\ref{1q1l1p1})}{\underset{\mathfrak{Z}_0}{\lesssim}} n^{-\hat
\theta}.
$$

Further we consider $q>2$.

The numbers $t_1(n)\le \hat t(n)$, $m_1(n)\le n^{q/2}$ will be
chosen later by $\mathfrak{Z}_0$. The numbers $l_{t,m}$ are
defined by (\ref{l_tm}).

Let
\begin{align}
\label{2_mt_st1} 2^{m_t^*} = \max \{2^{\hat m_t-\varepsilon
|t-t_1(n)|}, \, 1\}.
\end{align}
By (\ref{tn_hat}) and (\ref{2t}), for $t_0\le t\le t(n)$ and
sufficiently small $\varepsilon>0$ we have $2^{m_t^*}=2^{\hat
m_t-\varepsilon |t-t_1(n)|}$; hence,
\begin{align}
\label{rank_p} \sum \limits _{t_0\le t\le [t(n)]} {\rm rk}\,
P_{t,\lceil m_t^*\rceil} \stackrel{(\ref{rk_ptm}), (\ref{hat_mt}),
(\ref{2_mt_st1})} {\underset{\mathfrak{Z}_0}{\lesssim}} n.
\end{align}

For $f\in BX_{p_0}(\Omega)\cap BX_{p_1}(\Omega)$ we have
\begin{align}
\label{f_sum} f = \sum \limits _{t_0\le t\le [\hat t(n)]}
P_{t,\lceil m_t^*\rceil }f + \sum \limits _{t_0\le t\le [\hat
t(n)]} \sum \limits _{m\ge \lceil m_t^*\rceil} (P_{t,m+1}f -
P_{t,m}f) + f\cdot \chi _{\tilde \Omega _{[\hat t(n)]+1}}.
\end{align}
By (\ref{p0lq_emb}), (\ref{til_theta}) and (\ref{tn_hat}),
\begin{align}
\label{ost} \|f\|_{Y_q(\tilde \Omega_{[\hat t(n)]+1})}
\underset{\mathfrak{Z}_0}{\lesssim} n^{-q\hat \theta/2}.
\end{align}
From (\ref{dl}), (\ref{dl1}), (\ref{l_tm}), (\ref{rank_p}),
(\ref{f_sum}), (\ref{ost}) it follows that it suffices to estimate
the sum
$$
\sum \limits _{t_0\le t\le [\hat t(n)]} \sum \limits _{m\ge \lceil
m_t^*\rceil} d_{l_{t,m}} (W_{t,m}, \, l_q^{\nu_{t,m}})+ $$$$+\sum
\limits _{[t(n)]<t\le [\hat t(n)]} 2^{-\alpha_*k_*t}\cdot
2^{-m_t^*(1/q-1/p_0)}d_{l_{t,\lceil
m_t^*\rceil}}(B_{p_0}^{\nu'_{t,\lceil m_t^*\rceil}}, \,
l_q^{\nu'_{t,\lceil m_t^*\rceil}}) =:S_0.
$$

{\bf Case $\max\{p_0, \, p_1\}\le 2< q$.} We have $B_{p_0}^\nu
\subset B_2^\nu$, $B_{p_1}^\nu\subset B_2^\nu$; therefore, we get
by (\ref{w_tm})
$$
S_0\le \sum \limits _{0\le t\le t(n)} \sum \limits _{m\ge m_t^*}
2^{\mu_*k_*t}\cdot
2^{-m(s_*+1/q-1/p_1)}d_{l_{t,m}}(B_2^{\nu_{t,m}}, \,
l_q^{\nu_{t,m}}) +
$$
$$
+\sum \limits _{t(n)\le t\le \hat t(n)} \sum \limits _{m_t^*\le
m\le m_t}2^{-\alpha_*k_*t}\cdot
2^{-m(1/q-1/p_0)}d_{l_{t,m}}(B_2^{\nu_{t,m}}, \, l_q^{\nu_{t,m}})
+$$$$+\sum \limits _{t(n)\le t\le \hat t(n)} \sum \limits _{m>
m_t} 2^{\mu_*k_*t}\cdot
2^{-m(s_*+1/q-1/p_1)}d_{l_{t,m}}(B_2^{\nu_{t,m}}, \,
l_q^{\nu_{t,m}})=:S.
$$
Applying Theorem \ref{gl_teor} and taking into account
(\ref{rk_ptm}), we get that in the first and the third sums there
are decreasing geometric progressions in $m$, and in the second
sum there is an increasing geometric progression in $m$. Taking
into account (\ref{line_mt}) and (\ref{mt_t}), we get
$$
S\underset{\mathfrak{Z}_0}{\lesssim} \sum \limits _{0\le t\le
t(n)} \sum \limits _{m_t^*\le m\le
\overline{m}_t}2^{\mu_*k_*t}\cdot 2^{-m(s_*+1/q-1/p_1)} 2^{\gamma
k_*t/q}\cdot 2^{m/q}\cdot l_{t,m}^{-1/2}+
$$
$$
+ \sum \limits _{t(n)\le t\le\hat t(n)} \sum \limits
_{m=m_t}^{\overline{m}_t}2^{\mu_*k_*t}\cdot 2^{-m(s_*+1/q-1/p_1)}
2^{\gamma k_*t/q}\cdot 2^{m/q}\cdot l_{t,m}^{-1/2} =:\tilde S.
$$
Notice that $\theta_{j_*}<\min _{j\ne j_*}\theta_j$ by theorem
condition. Applying Lemma 6 from \cite{vas_bes} and taking into
account (\ref{hat_mt}), (\ref{line_mt}), (\ref{1111}),
(\ref{l_tm}), (\ref{2_mt_st1}), we get that for appropriate
$t_1(n)$ and $m_1(n)$ the estimate $\tilde S
\underset{\mathfrak{Z}_0}{\lesssim} S_1(n)+S_2(n)+ S_3(n) +S_4(n)$
holds with
$$
S_1(n) = n^{-s_*-1/q+1/p_1}\cdot n^{1/q-1/2} = n^{-s_*-1/2+1/p_1},
\quad S_2(n) = n^{-\frac{q}{2}(s_*+1/q-1/p_1)},
$$
$$
S_3(n) =2^{\mu_*k_*t(n)}\cdot 2^{\gamma_*(s_*+1/q-1/p_1)k_*t(n)}
\cdot n^{-s_*-1/q+1/p_1} \cdot n^{1/q-1/2}
\stackrel{(\ref{til_theta}),
(\ref{tn1})}{\underset{\mathfrak{Z}_0}{\lesssim}} n^{-\hat \theta
+1/q-1/2},
$$
$$
S_4(n) = 2^{\mu_*k_*\hat t(n)}\cdot
2^{\gamma_*(s_*+1/q-1/p_1)k_*\hat t(n)}\cdot n^{-\frac q2
(s_*+1/q-1/p_1)} \stackrel{(\ref{til_theta}),
(\ref{tn_hat})}{\underset{\mathfrak{Z}_0}{\lesssim}} n^{-q\hat
\theta/2}.
$$

{\bf Case $2\le p_0<q$, $2\le p_1\le q$.} We define the numbers
$m_t'$ by
\begin{align}
\label{mt_pr} \begin{array}{c} 2^{\mu_*k_*t}\cdot
2^{-m_t'(s_*+1/q-1/p_1)} \left(n^{-\frac 12}\cdot
2^{\frac{\gamma_*k_*t}{q}}\cdot 2^{\frac {m_t'}{q}}\right)
^{\frac{1/p_1-1/q}{1/2-1/q}} =
\\=2^{-\alpha_*k_*t}\cdot 2^{-m_t'(1/q-1/p_0)}\left(n^{-\frac
12}\cdot 2^{\frac{\gamma_*k_*t}{q}}\cdot 2^{\frac
{m_t'}{q}}\right)^{\frac{1/p_0-1/q}{1/2-1/q}}. \end{array}
\end{align}
Then
$$
2^{(\mu_*+\alpha_*)k_*t}\cdot
2^{-m_t'\left(s_*+1/p_0-1/p_1\right)} = \left(n^{-\frac 12}\cdot
2^{\frac{\gamma_*k_*t}{q}}\cdot 2^{\frac
{m_t'}{q}}\right)^{\frac{1/p_0-1/p_1}{1/2-1/q}};
$$
by (\ref{hat_mt}), (\ref{line_mt}), (\ref{1111}), (\ref{tn}),
(\ref{2222}), (\ref{tn_hat}), we have
\begin{align}
\label{mtn} m'_{\tilde t(n)}=\tilde m_{\tilde t(n)}=\hat m_{\tilde
t(n)}, \quad m'_{\hat t(n)} = m_{\hat t(n)} = \overline{m}_{\hat
t(n)}.
\end{align}

Applying Theorem \ref{gl_teor} together with (\ref{rk_ptm}) and
(\ref{w_tm}), we have
$$
S_0\underset{\mathfrak{Z}_0}{\lesssim} \sum \limits _{0\le t\le
\min\{\tilde t(n), \, \hat t(n)\}} \sum \limits _{m_t^*\le m\le
\overline{m}_t} 2^{\mu_*k_*t}\cdot 2^{-m(s_*+1/q-1/p_1)}
\left(l_{t,m}^{-\frac 12}\cdot 2^{\frac{\gamma_*k_*t}{q}}\cdot
2^{\frac mq}\right)^{\frac{1/p_1-1/q}{1/2-1/q}} +
$$
$$
+ \sum \limits _{\tilde t(n)<t\le \hat t(n)} \sum \limits
_{m^*_t\le m\le m_t'}2^{-\alpha_*k_*t}\cdot
2^{-m(1/q-1/p_0)}\left(l_{t,m}^{-\frac 12}\cdot
2^{\frac{\gamma_*k_*t}{q}}\cdot 2^{\frac
mq}\right)^{\frac{1/p_0-1/q}{1/2-1/q}} +
$$
$$
+ \sum \limits _{\tilde t(n)<t\le \hat t(n)} \sum \limits _{
m_t'\le m\le \overline{m}_t}2^{\mu_*k_*t}\cdot
2^{-m(s_*+1/q-1/p_1)} \left(l_{t,m}^{-\frac 12}\cdot
2^{\frac{\gamma_*k_*t}{q}}\cdot 2^{\frac
mq}\right)^{\frac{1/p_1-1/q}{1/2-1/q}}+
$$
$$
+\sum \limits _{0\le t\le \hat t(n)}\sum \limits _{m\ge
\overline{m}_t} 2^{\mu_*k_*t}\cdot 2^{-m(s_*+1/q-1/p_1)}=:S.
$$
In the second sum there is an increasing geometric progression in
$m$, in the last sum there is a decreasing geometric progression
in $m$. Applying Lemma 6 from \cite{vas_bes} and taking into
account (\ref{mt_pr}), (\ref{mtn}), we get that $S
\underset{\mathfrak{Z}_0}{\lesssim} S_1(n)+S_2(n)+S_3(n)+S_4(n)$,
where
$$
S_1(n)= n^{-s_*-1/q+1/p_1} \left(n^{-\frac 12}\cdot
 n^{\frac
1q}\right)^{\frac{1/p_1-1/q}{1/2-1/q}} = n^{-s_*}, \quad S_2(n)
=n^{-\frac q2 (s_*+1/q-1/p_1)},
$$
$$
S_3(n) = 2^{\mu_*k_*\tilde t(n)}\cdot
2^{(s_*+1/q-1/p_1)\gamma_*k_*\tilde t(n)}n^{-s_*-1/q+1/p_1}
\left(n^{-\frac 12}\cdot n^{\frac
1q}\right)^{\frac{1/p_1-1/q}{1/2-1/q}}
\stackrel{(\ref{til_theta}),
(\ref{tn})}{\underset{\mathfrak{Z}_0}{\lesssim}} n^{-\tilde
\theta},
$$
$$
S_4(n) =2^{\mu_*k_*\hat t(n)}\cdot 2^{-m'_{\hat
t(n)}(s_*+1/q-1/p_1)} \stackrel{(\ref{til_theta}), (\ref{tn_hat}),
(\ref{mtn})}{\underset{\mathfrak{Z}_0}{\lesssim}} n^{-\frac{q\hat
\theta}{2}}.
$$

{\bf Case $q>2$, $p_0<q$, $p_1\le q$, $\min\{p_0, \,
p_1\}<2<\max\{p_0, \, p_1\}$.} Let
\begin{align}
\label{12till} \frac 12 = \frac{1-\tilde\lambda}{p_1}
+\frac{\tilde\lambda}{p_0}.
\end{align}
Then (\ref{wkm_l2}) holds.

If $p_0>2>p_1$, the number $m_t'$ is defined by equation
\begin{align}
\label{p02p1}
\begin{array}{c}
2^{((1-\tilde\lambda)\mu_*-\tilde\lambda \alpha_*)k_*t}\cdot
2^{-((1-\tilde\lambda)(s_*+1/q-1/p_1)+\tilde\lambda(1/q-1/p_0))m_t'}
\cdot n^{-\frac 12}\cdot 2^{\frac{\gamma_*k_*t}{q}}\cdot
2^{\frac{m'_t}{q}} =
\\
= 2^{-\alpha_*k_*t}\cdot 2^{(1/p_0-1/q)m_t'} \left(n^{-\frac 12}
\cdot 2^{\frac{\gamma_*k_*t}{q}}\cdot 2^{\frac{m_t'}{q}}\right)
^{\frac{1/p_0-1/q}{1/2-1/q}};
\end{array}
\end{align}
if $p_1>2>p_0$, it is defined by equation
$$
\begin{array}{c}
2^{((1-\tilde\lambda)\mu_*-\tilde\lambda \alpha_*)k_*t}\cdot
2^{-((1-\tilde\lambda)(s_*+1/q-1/p_1)+\tilde\lambda(1/q-1/p_0))m_t'}
\cdot n^{-\frac 12}\cdot 2^{\frac{\gamma_*k_*t}{q}}\cdot
2^{\frac{m'_t}{q}} =
\\
= 2^{\mu_*k_*t}\cdot 2^{-m'_t(s_*+1/q-1/p_1)}\left(n^{-\frac 12}
\cdot 2^{\frac{\gamma_*k_*t}{q}}\cdot 2^{\frac{m_t'}{q}}\right)
^{\frac{1/p_1-1/q}{1/2-1/q}}.
\end{array}
$$
In both cases we get by (\ref{12till})
\begin{align}
\label{2p0p1} 2^{(\mu_*+\alpha_*)k_*t}\cdot
2^{-m_t'(s_*+1/p_0-1/p_1)}= \left(n^{-\frac 12}\cdot
2^{\frac{\gamma_*k_*t}{q}}\cdot 2^{\frac{m_t'}{q}}\right)
^{\frac{1/p_0-1/p_1}{1/2-1/q}};
\end{align}
from (\ref{hat_mt}), (\ref{line_mt}), (\ref{1111}), (\ref{tn}),
(\ref{2222}), (\ref{tn_hat}) we get that
\begin{align}
\label{mtiln} m'_{\tilde t(n)}=\hat m_{\tilde t(n)}=\tilde
m_{\tilde t(n)}, \quad m'_{\hat t(n)} = \overline{m}_{\hat t(n)}=
m_{\hat t(n)}.
\end{align}
Notice that the equality $s_*+\frac{1}{p_0}-\frac{1}{p_1} = \frac
1q \cdot \frac{\frac{1}{p_1}-\frac{1}{p_0}}{\frac 12-\frac1q}$
holds if and only if $\tilde t(n)=\hat t(n)$ (then $m_t'$ cannot
be defined by (\ref{2p0p1})).

Let us estimate the sum $S_0$.

For $p_0>2>p_1$ we define the subsets
$$
{\rm I}=\{(t, \, m):\; t\ge 0, \; m_t^*\le m\le \overline{m}_t, \;
m\ge m_t\},
$$
$$
{\rm II}=\{(t, \, m):\; 0\le t\le \hat t(n), \; m\ge
\overline{m}_t\},
$$
$$
{\rm III} =\{(t, \, m):\; t(n)\le t\le \hat t(n), \; m\ge m_t^*,
\; m\le m_t; \; m\ge m_t'\; \text{for}\; \tilde t(n)<t\le \hat
t(n)\},
$$
$$
{\rm IV} =\{(t, \, m):\; \tilde t(n)<t\le \hat t(n),\; m\ge m_t^*,
\; m\le m_t'\}.
$$

For $(t, \, m)\in {\rm I}\cup {\rm II}$ we apply the inclusion
$W_{t,m}\subset 2^{\mu_*k_*t}\cdot
2^{-(s_*+1/q-1/p_1)m}B_{p_1}^{\nu_{t,m}}$, for $(t, \, m)\in {\rm
III}$ we apply (\ref{wkm_l2}), for $(t, \, m)\in {\rm IV}$ we
apply the inclusion $W_{t,m} \subset 2^{-\alpha_*k_*t}\cdot
2^{-m(1/q-1/p_0)} B_{p_0}^{\nu_{t,m}}$. By Theorem \ref{gl_teor},
we get
$$
S_0 \underset{\mathfrak{Z}_0}{\lesssim} \sum \limits _{(t, \,
m)\in {\rm I}} 2^{\mu_*k_*t}\cdot 2^{-m(s_*+1/q-1/p_1)}
l_{t,m}^{-\frac{1}{2}}2^{\frac{\gamma_*k_*t}{q}}\cdot
2^{\frac{m}{q}}+ \sum \limits _{(t, \, m)\in {\rm II}}
2^{\mu_*k_*t}\cdot 2^{-m(s_*+1/q-1/p_1)} +
$$
$$
+ \sum \limits _{(t, \, m)\in {\rm III}}
2^{((1-\tilde\lambda)\mu_*-\tilde\lambda \alpha_*)k_*t}\cdot
2^{-((1-\tilde\lambda)(s_*+1/q-1/p_1) +\tilde\lambda(1/q-1/p_0))m}
\cdot l_{t,m}^{-\frac 12}\cdot 2^{\frac{\gamma_*k_*t}{q}}\cdot
2^{\frac{m}{q}} +
$$
$$
+ \sum \limits _{(t, \, m)\in {\rm IV}} 2^{-\alpha_*k_*t}\cdot
2^{(1/p_0-1/q)m} \left(l_{t,m}^{-\frac 12} \cdot
2^{\frac{\gamma_*k_*t}{q}}\cdot 2^{\frac{m}{q}}\right)
^{\frac{1/p_0-1/q}{1/2-1/q}}=:S.
$$
In the second sum there is a decreasing geometric progression in
$m$, in the fourth sum there is an increasing geometric
progression in $m$. Applying Lemma 6 from \cite{vas_bes} and
taking into account (\ref{hat_mt}), (\ref{line_mt}), (\ref{mt_t}),
(\ref{p02p1}), (\ref{2p0p1}), (\ref{mtiln}), we get that for
appropriate $t_1(n)$ and $m_1(n)$ the estimate
$S\underset{\mathfrak{Z}}{\lesssim} S_1(n) +
S_2(n)+S_3(n)+S_4(n)+S_5(n)$ holds with
$$
S_1(n) = n^{-(s_*+1/q-1/p_1)} n^{-1/2}\cdot n^{1/q} =
n^{-s_*-1/2+1/p_1},
$$
$$
S_2(n) = n^{-\frac{q}{2}(s_*+1/q-1/p_1)} n^{-1/2}\cdot
(n^{q/2})^{\frac 1q} = n^{-\frac{q}{2}(s_*+1/q-1/p_1)},
$$
$$
S_3(n) = 2^{\mu_*k_*t(n)}\cdot 2^{\gamma_*(s_*+1/q-1/p_1)k_*t(n)}
\cdot n^{-s_*-1/q+1/p_1} n^{-1/2+1/q} \stackrel{(\ref{til_theta}),
(\ref{tn1})}{\underset{\mathfrak{Z}_0}{\lesssim}}n^{-\hat \theta
+1/q-1/2},
$$
$$
S_4(n) = 2^{\mu_*k_*\hat t(n)}\cdot
2^{\gamma_*(s_*+1/q-1/p_1)k_*\hat t(n)} \cdot n^{-\frac{q}{2}
(s_*+1/q-1/p_1)} \stackrel{(\ref{til_theta}),
(\ref{tn_hat})}{\underset{\mathfrak{Z}_0}{\lesssim}} n^{-q\hat
\theta/2},
$$
$$
S_5(n) = 2^{-\alpha_*k_*\tilde t(n)}\cdot 2^{-\gamma_*
(1/p_0-1/q)k_*\tilde t(n)}\cdot n^{1/p_0-1/q} \cdot n^{1/q-1/p_0}
\stackrel{(\ref{til_theta}),
(\ref{tn})}{\underset{\mathfrak{Z}_0}{\lesssim}} n^{-\tilde
\theta}.
$$

For $p_1>2>p_0$ we set
$$
{\rm I}=\{(t, \, m):\; t\ge 0, \; m_t^*\le m\le \overline{m}_t, \;
m\ge m_t'\},
$$
$$
{\rm II}=\{(t, \, m):\; 0\le t\le \hat t(n), \; m\ge
\overline{m}_t\},
$$
$$
{\rm III} =\{(t, \, m):\; m\ge m_t^*, \; m\le m_t', \; m\ge m_t\},
$$
$$
{\rm IV} =\{(t, \, m):\; t\le \hat t(n),\; m\ge m_t^*, \; m\le
m_t\}.
$$

As in the previous case, we get that for appropriate $t_1(n)$ and
$m_1(n)$
$$
S_0 \underset{\mathfrak{Z}_0}{\lesssim} \sum \limits _{(t, \,
m)\in {\rm I}} 2^{\mu_*k_*t}\cdot 2^{-m(s_*+1/q-1/p_1)}
\left(l_{t,m}^{-1/2}2^{\gamma_*k_*t/q}\cdot
2^{m/q}\right)^{\frac{1/p_1-1/q}{1/2-1/q}}+ \sum \limits _{(t, \,
m)\in {\rm II}} 2^{\mu_*k_*t}\cdot 2^{-m(s_*+1/q-1/p_1)} +
$$
$$
+ \sum \limits _{(t, \, m)\in {\rm III}}
2^{((1-\tilde\lambda)\mu_*-\tilde\lambda \alpha_*)k_*t}\cdot
2^{-((1-\tilde\lambda)(s_*+1/q-1/p_1) +\tilde\lambda(1/q-1/p_0))m}
\cdot l_{t,m}^{-\frac 12}\cdot 2^{\frac{\gamma_*k_*t}{q}}\cdot
2^{\frac{m}{q}} +
$$
$$
+ \sum \limits _{(t, \, m)\in {\rm IV}} 2^{-\alpha_*k_*t}\cdot
2^{(1/p_0-1/q)m} l_{t,m}^{-\frac 12} \cdot
2^{\frac{\gamma_*k_*t}{q}}\cdot
2^{\frac{m}{q}}\underset{\mathfrak{Z}}{\lesssim}
$$
$$ \lesssim S_1(n) + S_2(n)+S_3(n)+S_4(n)+S_5(n),$$
where
$$
S_1(n) = n^{-(s_*+1/q-1/p_1)} \left(n^{-1/2}\cdot
n^{1/q}\right)^{\frac{1/p_1-1/q}{1/2-1/q}} = n^{-s_*},
$$
$$
S_2(n) = n^{-\frac{q}{2}(s_*+1/q-1/p_1)} \left(n^{-1/2}\cdot
(n^{q/2})^{\frac 1q}\right)^{\frac{1/p_1-1/q}{1/2-1/q}} =
n^{-\frac{q}{2}(s_*+1/q-1/p_1)},
$$
$$
S_3(n) = 2^{\mu_*k_*\tilde t(n)}\cdot
2^{\gamma_*(s_*+1/q-1/p_1)k_*\tilde t(n)} \cdot n^{-s_*-1/q+1/p_1}
n^{-1/p_1+1/q} \stackrel{(\ref{til_theta}),
(\ref{tn})}{\underset{\mathfrak{Z}_0}{\lesssim}} n^{-\tilde
\theta},
$$
$$
S_4(n) = 2^{\mu_*k_*\hat t(n)}\cdot
2^{\gamma_*(s_*+1/q-1/p_1)k_*\hat t(n)} \cdot n^{-\frac{q}{2}
(s_*+1/q-1/p_1)} \stackrel{(\ref{til_theta}),
(\ref{tn_hat})}{\underset{\mathfrak{Z}_0}{\lesssim}} n^{-q\hat
\theta/2},
$$
$$
S_5(n) = 2^{-\alpha_*k_* t(n)}\cdot 2^{-\gamma_*
(1/p_0-1/q)k_*t(n)}\cdot n^{1/p_0-1/q} \cdot n^{1/q-1/2}
\stackrel{(\ref{til_theta}),
(\ref{tn1})}{\underset{\mathfrak{Z}_0}{\lesssim}} n^{-\hat
\theta+1/q-1/2}.
$$

{\bf Case $q>2$, $2\le p_0<q$, $p_1>q$.} Let $\frac 1q =
\frac{1-\lambda}{p_1} +\frac{\lambda}{p_0}$. Then (\ref{wkm_lq})
holds.

We define the numbers $m_t'$ by equation
\begin{align}
\label{1mt_pr}
\begin{array}{c}
2^{-\alpha_*k_*t}\cdot 2^{m'_t(1/p_0-1/q)}\left(n^{-\frac
12}2^{\frac{\gamma_*k_*t}{q}}\cdot 2^{\frac{m'_t}{q}}\right)
^{\frac{1/p_0-1/q}{1/2-1/q}}=
\\
= 2^{((1-\lambda)\mu_*-\lambda\alpha_*)k_*t}\cdot
2^{-((1-\lambda)(s_*+1/q-1/p_1)+\lambda(1/q-1/p_0))m'_t}.
\end{array}
\end{align}
Then
\begin{align}
\label{1mt_pr1} 2^{(\mu_*+\alpha_*)k_*t}\cdot
2^{-m'_t(s_*+1/p_0-1/p_1)} = \left(n^{-\frac 12}\cdot
2^{\frac{\gamma_*k_*t}{q}}\cdot
2^{\frac{m_t'}{q}}\right)^{\frac{1/p_0-1/p_1}{1/2-1/q}};
\end{align}
by (\ref{hat_mt}), (\ref{line_mt}), (\ref{1111}), (\ref{tn}),
(\ref{2222}), (\ref{tn_hat}), we have
\begin{align}
\label{hat_m} \hat m_{\tilde t(n)}=\tilde m_{\tilde
t(n)}=m'_{\tilde t(n)}, \quad \overline{m}_{\hat t(n)}=m_{\hat
t(n)}= m'_{\hat t(n)}.
\end{align}

We define the number $\overline{t}(n)$ by equation $\tilde
m_{\overline{t}(n)}=\overline{m}_{\overline{t}(n)}$. We split the
set $\{(t, \, m):\; 0\le t\le \hat t(n), \; m\ge m_t^*\}$ into
subsets
$$
{\rm I} =\{(t, \, m):\; 0\le t\le \hat t(n), \; m_t^*\le m\le
\overline{m}_t, \; m\ge \tilde m_t\},
$$
$$
{\rm II} = \{(t, \, m):\; m\ge \overline{m}_t, \; 0\le t\le
\overline{t}(n)\},
$$
$$
{\rm III} =\{(t, \, m):\; m\le \tilde m_t, \; m\le \overline{m}_t,
\; m\ge m_t'\},
$$
$$
{\rm IV} =\{(t, \, m):\; \overline{t}(n)\le t\le \hat t(n), \;
m\ge \overline{m}_t\},
$$
$$
{\rm V} = \{(t, \, m):\; m\le m_t', \; m\ge m_t^*, \; t\le \hat
t(n)\}.
$$

For $(t, \, m)\in {\rm I}\cup {\rm II}$ we use the inclusion
$W_{t,m} \subset 2^{\mu_*k_*t}\cdot 2^{-(s_*+1/q-1/p_1)m}
B_{p_1}^{\nu_{t,m}}$, for $(t, \, m)\in {\rm III}\cup {\rm IV}$ we
apply (\ref{wkm_lq}), for $(t, \, m)\in {\rm V}$ we use the
inclusion $W_{t,m} \subset 2^{-\alpha_*k_*t}\cdot
2^{-(1/q-1/p_0)m} B_{p_0} ^{\nu_{t,m}}$.

By Theorem \ref{gl_teor} and (\ref{pietsch_stesin}), we get
$$
S_0\underset{\mathfrak{Z}_0}{\lesssim} \sum \limits _{(t, \, m)\in
{\rm I}\cup {\rm II}} 2^{\mu_*k_*t}\cdot
2^{-m(s_*+1/q-1/p_1)}(2^{\gamma_*k_*t}\cdot 2^m)^{1/q-1/p_1}+
$$
$$
+ \sum \limits _{(t, \, m)\in {\rm III}\cup {\rm IV}}
2^{((1-\lambda)\mu_*-\lambda\alpha_*)k_*t}\cdot
2^{-s_*(1-\lambda)m} +
$$
$$
+\sum \limits _{(t, \, m)\in {\rm V}} 2^{-\alpha_*k_*t}\cdot
2^{m(1/p_0-1/q)} \left(l_{t,m}^{-1/2}2^{\gamma_*k_*t/q}\cdot
2^{m/q}\right)^{\frac{1/p_0-1/q}{1/2-1/q}}=:S.
$$
Applying Lemma 6 from \cite{vas_bes}, (\ref{hat_mt}),
(\ref{line_mt}), (\ref{til_mt_t}), (\ref{1mt_pr}),
(\ref{1mt_pr1}), (\ref{hat_m}) and taking into account that in
${\rm I}\cup{\rm II}\cup{\rm III}\cup{\rm IV}$ there is a
decreasing geometric progression in $m$, and in ${\rm V}$ there is
an increasing geometric progression in $m$, we get that for
appropriate $t_1(n)$ and $m_1(n)$ the estimate
$S\underset{\mathfrak{Z}_0} {\lesssim} S_1(n) + S_2(n) + S_3(n)$
holds with
$$
S_1(n) = n^{-(s_*+1/q-1/p_1)}n^{1/q-1/p_1} = n^{-s_*},
$$
$$
S_2(n) = 2^{\mu_*k_*\tilde t(n)}\cdot
2^{\gamma_*(s_*+1/q-1/p_1)k_*\tilde t(n)}\cdot
n^{-s_*-1/q+1/p_1}\cdot n^{1/q-1/p_1} \stackrel{(\ref{til_theta}),
(\ref{tn})}{\underset{\mathfrak{Z}_0} {\lesssim}} n^{-\tilde
\theta},
$$
$$
S_3(n) =2^{-\alpha_*k_*\hat t(n)}\cdot
2^{\gamma_*(1/q-1/p_0)k_*\hat t(n)} \cdot
n^{(1/p_0-1/q)\frac{q}{2}} \stackrel{(\ref{til_theta}),
(\ref{tn_hat})}{\underset{\mathfrak{Z}_0}{\lesssim}} n^{-q\hat
\theta/2}.
$$

{\bf Case $p_0<2<q\le p_1$.} Let $\frac 1q=\frac{1-\lambda}{p_1} +
\frac{\lambda}{p_0}$, $\frac 12 = \frac{1-\tilde\lambda}{p_1}
+\frac{\tilde \lambda}{p_0}$. Then (\ref{wkm_lq}), (\ref{wkm_l2})
hold.

We define the numbers $m_t'$ by equation
\begin{align}
\label{ltill} \begin{array}{c} 2^{((1-\lambda)\mu_*-\lambda
\alpha_*)k_*t} \cdot
2^{-((1-\lambda)(s_*+1/q-1/p_1)+\lambda(1/q-1/p_0))m_t'} =
\\
=2^{((1-\tilde\lambda)\mu_*-\tilde\lambda \alpha_*)k_*t} \cdot
2^{-((1-\tilde\lambda)(s_*+1/q-1/p_1)+\tilde\lambda(1/q-1/p_0))m_t'}\cdot
n^{-\frac 12}\cdot 2^{\frac{m_t'}{q}}\cdot
2^{\frac{\gamma_*k_*t}{q}}.
\end{array}
\end{align}
Then
$$
2^{(\tilde \lambda-\lambda)(\mu_*+\alpha_*)k_*t}\cdot 2^{-(\tilde
\lambda-\lambda)(s_*+1/p_0-1/p_1)m_t'}= n^{-\frac 12}\cdot
2^{\frac{m_t'}{q}}\cdot 2^{\frac{\gamma_*k_*t}{q}};
$$
by (\ref{hat_mt}), (\ref{line_mt}), (\ref{1111}), (\ref{tn}),
(\ref{2222}), (\ref{tn_hat}), we have
\begin{align}
\label{2mpr} m'_{\tilde t(n)}=\hat m_{\tilde t(n)}=\tilde
m_{\tilde t(n)}, \quad m'_{\hat t(n)} = \overline{m}_{\hat t(n)}
=m _{\hat t(n)}.
\end{align}

We define the subsets ${\rm I}-{\rm IV}$ as in the previous case,
and set
$$
{\rm V}=\{(t, \, m):\; m\le m_t', \; m\ge m_t^*, \; m\ge m_t\},
$$
$$
{\rm VI} = \{(t, \, m):\; t\le \hat t(n), \; m\ge m_t^*, \; m\le
m_t\}.
$$

For $(t, \, m)\in {\rm I} \cup {\rm II}$ we use the inclusion
$W_{t,m} \subset 2^{\mu_*k_*t}\cdot 2^{-(s_*+1/q-1/p_1)m}
B_{p_1}^{\nu_{t,m}}$, for $(t, \, m)\in {\rm III}\cup {\rm IV}$ we
apply (\ref{wkm_lq}), for $(t, \, m)\in {\rm V}$ we apply
(\ref{wkm_l2}), in ${\rm VI}$ we apply the inclusion
$$W_{t,m} \subset 2^{-\alpha_*k_*t}\cdot 2^{-(1/q-1/p_0)m} B_{p_0}
^{\nu_{t,m}}.$$

We get
$$
S_0\underset{\mathfrak{Z}_0}{\lesssim} \sum \limits _{(t, \, m)\in
{\rm I}\cup {\rm II}} 2^{\mu_*k_*t}\cdot
2^{-m(s_*+1/q-1/p_1)}(2^{\gamma_*k_*t}\cdot 2^m)^{1/q-1/p_1}+
$$
$$
+ \sum \limits _{(t, \, m)\in {\rm III}\cup {\rm IV}}
2^{((1-\lambda)\mu_*-\lambda\alpha_*)k_*t}\cdot
2^{-s_*(1-\lambda)m} +
$$
$$
+\sum \limits _{(t, \, m)\in {\rm V}}
2^{((1-\tilde\lambda)\mu_*-\tilde\lambda\alpha_*)k_*t}\cdot
2^{-((1-\tilde\lambda)(s_*+1/q-1/p_1)+\tilde \lambda(1/q-1/p_0))m}
l_{t,m}^{-1/2}2^{\gamma_*k_*t/q}2^{m/q}+
$$
$$
+\sum \limits _{(t, \, m)\in {\rm VI}} 2^{-\alpha_*k_*t}\cdot
2^{m(1/p_0-1/q)} l_{t,m}^{-1/2}2^{\gamma_*k_*t/q}\cdot 2^{m/q}=:S.
$$

In the first and the second sums there is a decreasing geometric
progression in $m$, in the last sum there is an increasing
geometric progression in $m$. Applying Lemma 6 from \cite{vas_bes}
and taking into account (\ref{til_mt}), (\ref{mt}), (\ref{ltill}),
(\ref{2mpr}), we get that for appropriate $t_1(n)$ and $m_1(n)$
the estimate $S\underset{\mathfrak{Z}_0}{\lesssim} S_1(n)+S_2(n)
+S_3(n) + S_4(n)$ holds with
$$
S_1(n) = n^{-s_*-1/q+1/p_1}\cdot n^{1/q-1/p_1}= n^{-s_*},
$$
$$
S_2(n) = 2^{\mu_*k_*\tilde t(n)}\cdot
2^{\gamma_*(s_*+1/q-1/p_1)k_*\tilde t(n)}\cdot
n^{-s_*-1/q+1/p_1}\cdot n^{1/q-1/p_1} \stackrel{(\ref{til_theta}),
(\ref{tn})}{\underset{\mathfrak{Z}_0} {\lesssim}} n^{-\tilde
\theta},
$$
$$
S_3(n) =2^{-\alpha_*k_*\hat t(n)}\cdot 2^{-(1/p_0-1/q)\gamma_*k_*
\hat t(n)}\cdot n^{-\frac{q}{2}(1/q-1/p_0)} n^{-1/2}\cdot n^{\frac
q2\cdot \frac 1q} \stackrel{(\ref{til_theta}), (\ref{tn_hat})}
{\underset{\mathfrak{Z}_0} {\lesssim}} n^{-q\hat \theta/2},
$$
$$
S_4(n)=2^{-\alpha_*k_*t(n)}\cdot 2^{\gamma_*(1/q-1/p_0)k_*t(n)}
\cdot n^{1/p_0-1/q} n^{-1/2}\cdot n^{1/q}
\stackrel{(\ref{til_theta}), (\ref{tn1})}
{\underset{\mathfrak{Z}_0} {\lesssim}} n^{-\hat \theta +1/q-1/2}.
$$
This completes the proof.
\end{proof}

\section{Lower estimates for widths of $BX_{p_1}(\Omega)\cap BX_{p_0}(\Omega)$.}

Let $c\ge 1$, $t_0\in \Z_+$. Suppose that for all integers $t\ge
t_0$, $m\in \Z_+$ there are functions $\varphi_j^{t,m}\in
X_{p_0}(\Omega) \cap X_{p_1}(\Omega)$ ($1\le j\le \nu_{t,m}$) with
pairwise disjoint supports such that
\begin{align}
\label{nu_tm} \nu_{t,m} \ge c^{-1}2^{\gamma_*k_*t}\cdot 2^m
=:\nu'_{t,m},
\end{align}
\begin{align}
\label{func_est} \begin{array}{c} \|\varphi_j^{t,m}\|
_{Y_q(\Omega)} = 1, \quad \|\varphi _j^{t,m}\| _{X_{p_0}(\Omega)}
\le c\cdot 2^{\alpha_*k_*t}\cdot2^{m\left(\frac 1q
-\frac{1}{p_0}\right)}, \\ \|\varphi _j^{t,m}\| _{X_{p_1}(\Omega)}
\le c\cdot 2^{-\mu_*k_*t}\cdot 2^{m\left(s_*+1/q-1/p_1\right)}.
\end{array}
\end{align}
We denote $\mathfrak{Z}_1 = (c, \, t_0, \, q, \, p_0, \, p_1, \,
s_*, \gamma_*, \, \alpha_*, \, \mu_*)$. The numbers $\tilde
\theta$ and $\hat \theta$ are defined by (\ref{til_theta}).
\begin{Trm}
\label{low_est} Let (\ref{s1qp}), (\ref{mua}), (\ref{nu_tm}),
(\ref{func_est}) hold. Then
\begin{align}
\label{low_est1} d_n(BX_{p_0}(\Omega)\cap BX _{p_1}(\Omega), \,
Y_q(\Omega)) \underset{\mathfrak{Z}_1}{\gtrsim} n^{-s_*-1/q+1/p_1}
d_n(B_{p_1}^{2n}, \, l_q^{2n}),
\end{align}
\begin{align}
\label{low_est2} d_n(BX_{p_0}(\Omega)\cap BX _{p_1}(\Omega), \,
Y_q(\Omega)) \underset{\mathfrak{Z}_1}{\gtrsim} n^{-\tilde
\theta},
\end{align}
\begin{align}
\label{low_est3} d_n(BX_{p_0}(\Omega)\cap BX _{p_1}(\Omega), \,
Y_q(\Omega)) \underset{\mathfrak{Z}_1}{\gtrsim} n^{-\hat
\theta-(1/2-1/q)_+};
\end{align}
if $q>2$, $p_1<q$, then
\begin{align}
\label{low_est4} d_n(BX_{p_0}(\Omega)\cap BX _{p_1}(\Omega), \,
Y_q(\Omega)) \underset{\mathfrak{Z}_1}{\gtrsim}
n^{-q(s_*+1/q-1/p_1)/2};
\end{align}
if $q>2$, then
\begin{align}
\label{low_est5} d_n(BX_{p_0}(\Omega)\cap BX _{p_1}(\Omega), \,
Y_q(\Omega)) \underset{\mathfrak{Z}_1}{\gtrsim} n^{-q\hat
\theta/2}.
\end{align}
\end{Trm}
\begin{proof}
Let $\nu'_{t,m}\ge 2n$, $L= {\rm span}\, \{\varphi_1^{t,m}, \,
\dots, \, \varphi_{\nu_{t,m}}^{t,m}\}$, and let $W_{t,m}$ be the
set of sequences $(c_1, \, \dots, \, c_{\nu_{t,m}})\in
\R^{\nu_{t,m}}$ such that
$$
\|(c_j)_{j=1}^{\nu_{t,m}}\|_{l_{p_0}^{\nu _{t,m}}} \le
2^{-\alpha_*k_*t}\cdot 2^{m(1/p_0-1/q)}, \quad
\|(c_j)_{j=1}^{\nu_{t,m}}\|_{l_{p_1}^{\nu _{t,m}}} \le
2^{\mu_*k_*t}\cdot 2^{-m(s_*+1/q-1/p_1)},
$$
\begin{align}
\label{m_set} M= \left\{\sum \limits _{j=1}^{\nu_{t,m}}
c_j\varphi_j^{t,m}:\; (c_j)_{j=1}^{\nu_{t,m}}\in W_{t,m}\right\}.
\end{align}

Since the functions $\varphi_j^{t,m}$ have pairwise disjoint
supports, there is a linear projection from $Y_q(\Omega)$ onto $L$
with unit norm. It follows from the properties of the Kolmogorov
widths that
$$
d_n(BX_{p_0}(\Omega)\cap BX_{p_1}(\Omega), \, Y_q(\Omega))
\stackrel{(\ref{func_est}),(\ref{m_set})}{\underset{\mathfrak{Z}_1}{\gtrsim}}
d_n(M, \, Y_q(\Omega)) = $$$$=d_n(M, \, L)
\stackrel{(\ref{func_est}),(\ref{m_set})} {=} d_n(W_{t,m}, \,
l_q^{\nu_{t,m}}).
$$

Let $t=t_0$. By (\ref{s1qp}), there is $\tilde c=\tilde
c(\mathfrak{Z}_1)$ such that for sufficiently large $n$ and for
$\nu'_{t_0,m}\ge 2n$ the inclusion $W_{t_0,m} \supset \tilde c
\cdot 2^{-m(s_*+1/q-1/p_1)} B_{p_1}^{\nu_{t_0,m}}$ holds. We take
$m$ such that $2n\le \nu'_{t_0,m}
\underset{\mathfrak{Z}_1}{\lesssim} n$ or $2n^{q/2}\le
\nu'_{t_0,m} \underset{\mathfrak{Z}_1}{\lesssim} n^{q/2}$, apply
(\ref{nu_tm}) and Theorem \ref{gl_teor}, and obtain
(\ref{low_est1}) and (\ref{low_est4}).

The set $\min \{2^{-\alpha_*k_*t}\cdot 2^{m(1/p_0-1/q)}, \,
2^{\mu_*k_*t}\cdot 2^{-m(s_*+1/q-1/p_1)}\} B_1^{\nu_{t,m}}$ is
contained in $W_{t,m}$. Let $m_t$ be defined by equation
\begin{align}
\label{2akt} 2^{-\alpha_*k_*t}\cdot 2^{m_t(1/p_0-1/q)} =
2^{\mu_*k_*t}\cdot 2^{-m_t(s_*+1/q-1/p_1)}.
\end{align}
Then
$$
d_n(W_{t,[m_t]}, \, l_q^{\nu_{t,[m_t]}})
\underset{\mathfrak{Z}_1}{\gtrsim} 2^{\mu_*k_*t}\cdot
2^{-m_t(s_*+1/q-1/p_1)} d_n(B_1^{\nu_{t,[m_t]}}, \,
l_q^{\nu_{t,[m_t]}}).
$$
We take $t(n)$ and $\hat t(n)$ such that $2n\le
\nu'_{t(n),[m_{t(n)}]} \underset{\mathfrak{Z}_1}{\lesssim} n$ and
$2n^{q/2}\le \nu'_{\hat t(n),[m_{\hat t(n)}]}
\underset{\mathfrak{Z}_1}{\lesssim} n^{q/2}$. Applying
(\ref{nu_tm}) and (\ref{2akt}), we get (\ref{low_est3}) and
(\ref{low_est5}).

The set $\min \{2^{-\alpha_*k_*t}\cdot 2^{-\gamma_*k_*t/p_0}\cdot
2^{-m/q}, \, 2^{\mu_*k_*t}\cdot 2^{-\gamma_*k_*t/p_1}\cdot
2^{-m(s_*+1/q)}\} B_\infty^{\nu_{t,m}}$ is contained in $W_{t,m}$.
Let
\begin{align}
\label{akt2} 2^{-\alpha_*k_*t}\cdot 2^{-\gamma_*k_*t/p_0}\cdot
2^{-\tilde m_t/q} = 2^{\mu_*k_*t}\cdot2^{-\gamma_*k_*t/p_1}\cdot
2^{-\tilde m_t(s_*+1/q)}.
\end{align}
Then
$$
d_n(W_{t,[\tilde m_t]}, \, l_q^{\nu_{t,[\tilde m_t]}})
\underset{\mathfrak{Z}_1}{\gtrsim} 2^{-\alpha_*k_*t}\cdot
2^{-\gamma_*k_*t/p_0}\cdot 2^{-\tilde m_t/q}
 d_n(B_\infty^{\nu_{t,[\tilde m_t]}}, \,
l_q^{\nu_{t,[\tilde m_t]}})
\stackrel{(\ref{nu_tm})}{\underset{\mathfrak{Z}_1}{\gtrsim}}
$$
$$
\gtrsim 2^{-(\alpha_*+\gamma_*/p_0-\gamma_*/q)k_*t}.
$$
We take $\tilde t(n)$ such that $2n\le \nu'_{\tilde t(n),[\tilde
m_{\tilde t(n)}]} \underset{\mathfrak{Z}_1}{\lesssim} n$. Applying
(\ref{nu_tm}) and (\ref{akt2}), we get (\ref{low_est2}).
\end{proof}

\section{Upper estimates for widths of weighted Sobolev classes}

As $X_{p_1}(\Omega)$ we take the space ${\cal
W}^r_{p_1,g}(\Omega)$ with seminorm $\|f\|_{{\cal
W}^r_{p_1,g}(\Omega)} = \left\|\frac{\nabla^r f}{g}\right\|
_{L_{p_1}(\Omega)}$, as $X_{p_0}(\Omega)$ we take
$L_{p_0,w}(\Omega)$ with norm
$\|f\|_{L_{p_0,w}(\Omega)}=\|wf\|_{L_{p_0}(\Omega)}$, as
$Y_q(\Omega)$ we take $L_{q,v}(\Omega)$ with norm
$\|f\|_{L_{q,v}(\Omega)}=\|vf\|_{L_q(\Omega)}$, as ${\cal
P}(\Omega)$ we take the space ${\cal P}_{r-1}(\Omega)$ of
polynomials of degree at most $r-1$.

First we consider the function classes from Theorems \ref{main1},
\ref{main2}.

Let $\Omega\in {\bf FC}(a)$, and let $\Gamma \subset \partial
\Omega$ be an $h$-set. In \cite{vas_vl_raspr},
\cite{vas_vl_raspr2} the numbers $\overline{s}=\overline{s}(a, \,
d)\in \N$ and $b_*=b_*(a, \, d)>0$ were defined and the partition
of the domain $\Omega$ into subdomains $\Omega [\eta_{j,i}]\in
{\bf FC}(b_*)$ ($j\ge j_{\min}$, $i\in \tilde I_j$) was
constructed, such that
\begin{align}
\label{diam_om_ji} {\rm diam}\, \Omega [\eta_{j,i}]
\underset{a,\,d}{\asymp} 2^{-\overline{s}j}; \quad {\rm dist}\,
(x, \, \Gamma) \underset{a,\,d}{\asymp} 2^{-\overline{s}j}, \quad
x\in \Omega [\eta_{j,i}]; \quad {\rm card}\, \tilde I_j
\underset{a,\, d, \, c_*}{\lesssim}
\frac{h(2^{-\overline{s}j_{\min}})} {h(2^{-\overline{s}{j}})}.
\end{align}

{\bf Upper estimate in Theorem \ref{main1}.} We set $\hat J_t =
\tilde I_t$, $\Omega _{t,i}= \Omega[\eta_{t,i}]$. Then Assumption
\ref{supp1} follows from the Sobolev embedding theorem
\cite{resh1}, \cite{resh2} (since $r +\frac dq-\frac{d}{p_1}>0$
and $\Omega _{t,i}\in {\bf FC}(b_*)$). Assumption \ref{supp2} with
$\gamma_* = \theta$, $k_*=\overline{s}$ follows from
(\ref{h_theta}) and (\ref{diam_om_ji}). From \cite[Lemma
8]{vas_width_raspr} we get Assumption \ref{supp3}; the same lemma
together with (\ref{gw}), (\ref{diam_om_ji}) yield (\ref{fpef})
with $\mu_* = \beta+\lambda -r -\frac{d}{q} +\frac{d}{p_1}$.
Assumption \ref{supp4} with $\alpha_*=\sigma-\lambda +\frac dq
-\frac{d}{p_0}$ follows from H\"{o}lder's inequality, (\ref{gw})
and (\ref{diam_om_ji}). It remains to check (\ref{pef}). Notice
that the elements of the partition $T_{t,j,m}$ belong to ${\bf
FC}(b_*)$; it follows from \cite[Lemma 7]{vas_john} (Lemma 8 from
\cite{vas_width_raspr} is the corollary of Lemma 7 from
\cite{vas_john}). Hence, there are concentric balls $B_E\subset
\Omega _{t,i} \subset \tilde B_E$ of radii $R_E
\underset{a,d}{\asymp} 2^{-\overline{s}t}\cdot 2^{-\frac md}$ and
$\tilde R_E \underset{a,d}{\asymp} 2^{-\overline{s}t}\cdot
2^{-\frac md}$, respectively. The operator $P_E$ is defined as
follows: first the orthogonal projection in $L_2(B_E)$ onto ${\cal
P}_{r-1}(B_E)$ is constructed, then the polynomials are extended
onto $E$. We have
$$
\|P_Ef\|_{L_q(E)} \le \|P_Ef\|_{L_q(\tilde B_E)}
\underset{a,d,q,r}{\lesssim} \|P_Ef\|_{L_q(B_E)}
\underset{a,q,d,r}{\lesssim} 2^{(1-1/q)(\overline{s}td+m)}
\|P_Ef\|_{L_1(B_E)} \underset{d,r}{\lesssim}
$$
$$
\le 2^{(1-1/q)(\overline{s}td+m)} \|f\|_{L_1(B_E)}
\underset{a,d,p_0}{\lesssim}
2^{(1/p_0-1/q)(\overline{s}td+m)}\|f\|_{L_{p_0}(B_E)}\le
2^{(1/p_0-1/q)(\overline{s}td+m)}\|f\|_{L_{p_0}(E)}.
$$
We apply (\ref{gw}) together with (\ref{diam_om_ji}) and obtain
(\ref{pef}).

{\bf Upper estimate in Theorem \ref{main2}.} We set
$\{\Omega_{t,i}\} _{i\in \hat J_t} = \{\Omega [\eta_{j,i}]\}
_{2^t\le \overline{s}j<2^{t+1},i\in \tilde I_j}$. Arguing as in
the proof of Theorem \ref{main1} and applying (\ref{gvw}),
(\ref{h_gam}), (\ref{gw_1}), (\ref{lim_case}), (\ref{diam_om_ji}),
we get Assumptions \ref{supp1}--\ref{supp5} with $k_*=1$,
$\gamma_*=\gamma+1$, $\alpha_*=\alpha-\nu$, $\mu_*=\mu+\nu$.

{\bf Upper estimate in Theorem \ref{main3}.} We set $\Omega_0 =
(-1, \, 1)^d$; for $t\in \N$ we write $\Omega _t =(-2^t, \, 2^t)^d
\backslash [-2^{t-1}, \, 2^{t-1}]^d$; let $\hat J_t=\{1\}$. As in
the proofs of the previous theorems, we get by (\ref{gw_2})
Assumptions \ref{supp1}--\ref{supp5} with $\gamma_*=0$,
$\mu_*=\beta+\lambda+r+\frac dq-\frac{d}{p_1}$,
$\alpha_*=\sigma-\lambda+\frac{d}{p_0}-\frac dq$.

\section{Lower estimates for widths of weighted Sobolev classes}

Let $\psi \in C_0^\infty(\R^d)$, $\|\psi\|_{L_q(\R^d)}=1$, ${\rm
supp}\, \psi \subset [0, \, 1]^d$.

In the first and the second examples the functions $g$, $v$, $w$
have the form $g(x)=\varphi_g({\rm dist}\, (x, \, \Gamma))$, $w(x)
=\varphi_w({\rm dist}\, (x, \, \Gamma))$, $v(x) =\varphi_v({\rm
dist}\, (x, \, \Gamma))$.

In \cite[p. 118]{vas_vl_raspr2} the number
$k_{**}=k_{**}(\mathfrak{Z})$ was defined and the family of cubes
$\{\Delta_{j,i}\}_{i\in I_j}$ was constructed, such that
$\Delta_{j,i}\subset \Omega$, ${\rm mes}\, \Delta_{j,i}
\underset{a,d,c_*}{\asymp} 2^{-dk_{**}j}$,
\begin{align}
\label{ij} {\rm card}\, I_j \underset{a, \, d, \, c_*}{\gtrsim}
\frac{h(2^{-k_{**}j_{\min}})}{h(2^{-k_{**}j})}
\end{align}
(here the number $j_{\min}\in \Z_+$ depends on ${\rm
diam}\,\Omega$, $c_*$ is from Definition \ref{h_set}),
\begin{align}
\label{dist} {\rm dist}\, (x, \, \Gamma) \underset{a, \, d, \,
c_*}{\asymp} 2^{-k_{**}j}, \quad x\in \Delta _{j,i}.
\end{align}
For each $m\in \Z_+$ we take a partition of $\Delta_{j,i}$ into
cubes $\Delta_{j,i,l} = x_{j,i,l} + \rho _{j,i,m}[0, \, 1]^d$,
$1\le l\le l_m$,
\begin{align}
\label{ltm} l_m \underset{d}{\asymp} 2^m, \quad \rho _{j,i,m}
\underset{a,d,c_*}{\asymp} 2^{-k_{**}j}\cdot 2^{-m/d}.
\end{align}

Let $\psi _{j,i,l}(x) = c_{j,i,l}\psi\left(\frac{x-x_{j,i,l}}
{\rho_{j,i,m}}\right)$, where $c_{j,i,l}>0$ is such that
\begin{align}
\label{psi_jil} \|\psi _{j,i,l}\|_{L_{q,v}(\Delta_{j,i,l})} = 1.
\end{align}
Then by (\ref{dist}), (\ref{ltm}) we have
\begin{align}
\label{jm1} \|\psi_{j,i,l}\| _{{\cal W}^r_{p_1,g}(\Delta_{j,i,l})}
\underset{\mathfrak{Z}_*}{\asymp} 2^{(r +d/q-d/p_1)k_{**}j}\cdot
2^{m(r/d +1/q-1/p_1)}\cdot
\frac{1}{\varphi_g(2^{-k_{**}j})\varphi_v(2^{-k_{**}j})},
\end{align}
\begin{align}
\label{jm0} \|\psi_{j,i,l}\| _{L_{p_0,w}(\Delta_{j,i,l})}
\underset{\mathfrak{Z}_*}{\asymp} 2^{(d/q-d/p_0)k_{**}j} \cdot
2^{m(1/q-1/p_0)}\cdot \frac{\varphi_w(2^{-k_{**}j})}
{\varphi_v(2^{-k_{**}j})}.
\end{align}

{\bf Lower estimate in Theorem \ref{main1}.} We take
$\{\varphi^{t,m}_\nu\} _{1\le \nu\le \nu_{t,m}} =\{\psi
_{t,i,l}\}_{i\in I_t, \, 1\le l\le l_m}$. By (\ref{h_theta}),
(\ref{ij}) and (\ref{ltm}), $\nu_{t,m}
\underset{\mathfrak{Z}_*}{\gtrsim} 2^{\theta k_{**}t}\cdot 2^m$.
Hence, (\ref{nu_tm}) holds with $\gamma_*=\theta$, $k_*=k_{**}$.
By (\ref{gw}), (\ref{psi_jil}), (\ref{jm1}) and (\ref{jm0}), we
get (\ref{func_est}) with $\alpha_* = \sigma-\lambda +\frac dq
-\frac{d}{p_0}$, $\mu_*=\beta +\lambda - r -\frac dq+
\frac{d}{p_1}$, $s_*=\frac rd$. It remains to apply Theorem
\ref{low_est}.

{\bf Lower estimate in Theorem \ref{main2}.} We take
$\{\varphi^{t,m}_\nu\} _{1\le \nu\le \nu_{t,m}} =\{\psi
_{j,i,l}\}_{2^t\le j< 2^{t+1}, \, i\in I_j, \, 1\le l\le l_m}$. By
(\ref{h_gam}), (\ref{ij}) and (\ref{ltm}), $\nu_{t,m}
\underset{\mathfrak{Z}_*}{\gtrsim} 2^{(\gamma +1)t}\cdot 2^m$.
Hence, (\ref{nu_tm}) holds with $\gamma_*=\gamma+1$, $k_*=1$. By
(\ref{gw_1}), (\ref{lim_case}), (\ref{psi_jil}), (\ref{jm1}) and
(\ref{jm0}), we have (\ref{func_est}) with $\alpha_* =
\alpha-\nu$, $\mu_*=\mu+\nu$, $s_*=\frac rd$.

{\bf Lower estimate in Theorem \ref{main3}.} We set $\Omega_0 =
(0, \, 1)^d$, $$\Omega _t =(2^{t-1}, \, 2^t)^d, \quad t\in \N;$$
let $\{\Delta _{t,l}^m\}_{1\le l\le 2^{[m/d]}}$ be the partition
of $\Omega_t$ into cubes with side length $2^{t-1+[m/d]}$,
$\Delta_{t,l}^m = x_{t,l}^m + \rho _{t}^m(0, \, 1)^d$, $\psi
_{t,l}^m(x)=c_{t,l}^m\psi\left(\frac{x-x_{t,l}^m}
{\rho_{t}^m}\right)$, where $c_{t,l}^m>0$ is such that
$\|\psi_{t,l}^m\|_{L_{q,v}(\Delta_{t,l}^m)}=1$. Applying
(\ref{gw_2}), we get (\ref{nu_tm}) with $\gamma_*=0$, $k_*=1$ and
(\ref{func_est}) with $\mu_* =
\beta+\lambda+r+\frac{d}{q}-\frac{d}{p_1}$,
$\alpha_*=\sigma-\lambda+\frac{d}{p_0}-\frac dq$, $s_*=\frac rd$.


\begin{Biblio}
\bibitem{ait_kus1} M.S.~Aitenova, L.K.~Kusainova, ``On the asymptotics of the distribution of approximation
numbers of embeddings of weighted Sobolev classes. I'', {\it Mat.
Zh.} {\bf 2}:1 (2002), 3--9.

\bibitem{ait_kus2} M.S.~Aitenova, L.K.~Kusainova, ``On the asymptotics of the distribution of approximation
numbers of embeddings of weighted Sobolev classes. II'', {\it Mat.
Zh.} {\bf 2}:2 (2002), 7--14.

\bibitem{besov} O.V.~Besov, ``Sobolev's embedding theorem for a domain with irregular
boundary'', {\it Sb. Math.} {\bf 192}:3 (2001), 323--346.

\bibitem{boy_1} I.V.~Boykov, ``Approximation of some classes
of functions by local splines'', {\it Comput. Math. Math. Phys.}
{\bf 38}:1 (1998), 21--29.


\bibitem{m_bricchi1} M. Bricchi, ``Existence and properties of
$h$-sets'', {\it Georgian Mathematical Journal}, {\bf 9}:1 (2002),
13–-32.

\bibitem{caso_ambr}  L. Caso, R. D'Ambrosio, ``Weighted spaces and weighted norm inequalities on irregular
domains'', {\it J. Appr. Theory}, {\bf 167} (2013), 42--58.

\bibitem{edm_lang} D.\,E.~Edmunds, J.~Lang,
``Approximation numbers and Kolmogorov widths of Hardy-type
operators in a non-homogeneous case'', {\it Math. Nachr.} {\bf
297}:7 (2006), 727--742.

\bibitem{el_kolli} A. El Kolli, ``n-i\`{e}me \'{e}paisseur dans les espaces de Sobolev'',
{\it J. Approx. Theory}, {\bf 10} (1974), 268--294.

\bibitem{galeev1} E.M.~Galeev, ``The Kolmogorov diameter of the intersection of classes of periodic
functions and of finite-dimensional sets'', {\it Math. Notes},
{\bf 29}:5 (1981), 382--388.

\bibitem{bib_gluskin} E.D. Gluskin, ``Norms of random matrices and diameters
of finite-dimensional sets'', {\it Math. USSR-Sb.}, {\bf 48}:1
(1984), 173--182.

\bibitem{konov_leviat} V.\,N.~Konovalov, D.~Leviatan,
``Kolmogorov and linear widths of weighted Sobolev-type classes on
a finite interval''. {\it Anal. Math.} {\bf 28}:4 (2002),
251--278.

\bibitem{kufner} A. Kufner, {\it Weighted Sobolev spaces}. Teubner-Texte Math., 31.
Leipzig: Teubner, 1980.

\bibitem{kus1} L.K. Kusainova, ``Embedding the weighted Sobolev space $W^l_p(\Omega; v)$ in the space $L_p(\Omega; \omega)$''.
{\it Sb. Math.} {\bf 191}:2 (2000), 275--290.

\bibitem{lang_j_at1} J.~Lang, ``Improved estimates for the
approximation numbers of Hardy-type operators''. {\it J. Appr.
Theory}, {\bf 121}:1 (2003), 61--70.

\bibitem{lif_linde} M.\,A.~Lifshits, W.~Linde,
``Approximation and entropy numbers of Volterra operators with
application to Brownian motion'', {\it Mem. Amer. Math. Soc.} {\bf
157}, issue 745.

\bibitem{lo1} P.I. Lizorkin, M. Otelbaev, ``Imbedding theorems and compactness
for spaces of Sobolev type with weights'', {\it Math. USSR-Sb.}
{\bf 36}:3 (1980), 331--349.

\bibitem{lo2} P.I. Lizorkin, M. Otelbaev, ``Imbedding theorems and
compactness for spaces of Sobolev type with weights. II'', {\it
Math. USSR-Sb.} {\bf 40}:1 (1981), 51--77.

\bibitem{lo3} P.I. Lizorkin, M. Otelbaev, ``Estimates of approximate numbers
of the imbedding operators for spaces of Sobolev type with
weights'', {\it Proc. Steklov Inst. Math.}, {\bf 170} (1987),
245--266.

\bibitem{step_lom} E.N. Lomakina, V.D. Stepanov, ``Asymptotic Estimates for the Approximation
and Entropy Numbers of a One-Weight Riemann–Liouville Operator'',
{\it Siberian Adv. Math.}, {\bf 17}:1 (2007), 1--36.

\bibitem{mieth1} T. Mieth, ``Entropy and approximation numbers of embeddings of
weighted Sobolev spaces'', {\it J. Appr. Theory}, {\bf 192}
(2015), 250--272.

\bibitem{mieth2} T. Mieth, ``Entropy and approximation numbers of weighted
Sobolev spaces via bracketing'', {\it J. Funct. Anal.} {\bf 270}
(2016), 4322--4339.

\bibitem{myn_otel} K.~Mynbaev, M.~Otelbaev, {\it Weighted function spaces and the
spectrum of differential operators.} Nauka, Moscow, 1988.

\bibitem{r_oinarov} R. Oinarov, ``On weighted norm inequalities with three
weights''. {\it J. London Math. Soc.} (2), {\bf 48} (1993),
103--116.

\bibitem{pietsch1} A. Pietsch, ``$s$-numbers of operators in Banach space'', {\it Studia Math.},
{\bf 51} (1974), 201--223.

\bibitem{resh1} Yu.G. Reshetnyak, ``Integral representations of
differentiable functions in domains with a nonsmooth boundary'',
{\it Sibirsk. Mat. Zh.}, {\bf 21}:6 (1980), 108--116 (in Russian).

\bibitem{resh2} Yu.G. Reshetnyak, ``A remark on integral representations
of differentiable functions of several variables'', {\it Sibirsk.
Mat. Zh.}, {\bf 25}:5 (1984), 198--200 (in Russian).

\bibitem{st_ush} V.D. Stepanov, E.P. Ushakova, ``On Integral Operators with Variable Limits of Integration'',
{\it Proc. Steklov Inst. Math.}, {\bf 232} (2001), 290--309.

\bibitem{stepanov2} V. D. Stepanov, ``Two-weighted estimates of Riemann–Liouville integrals'', {\it Math. USSR-Izv.}, {\bf 36}:3 (1991),
669--681.

\bibitem{stepanov1} V.D. Stepanov, ``Weighted norm inequalities for integral operators and related topics''.
Nonlinear analysis, function spaces and applications, Vol. 5
(Prague, 1994), 139–175, Prometheus, Prague, 1994.

\bibitem{stesin} M.I. Stesin, ``Aleksandrov diameters of finite-dimensional sets
and of classes of smooth functions'', {\it Dokl. Akad. Nauk SSSR},
{\bf 220}:6 (1975), 1278--1281 [Soviet Math. Dokl.].

\bibitem{trieb_mat_sb} H. Triebel, ``Interpolation properties of $\varepsilon$-entropy and diameters.
Geometric characteristics of imbedding for function spaces of
Sobolev–Besov type'', {\it Math. USSR-Sb.}, {\bf 27}:1 (1975),
23--37.

\bibitem{triebel} H. Triebel, {\it Interpolation theory. Function spaces. Differential
operators}. Mir, Moscow, 1980.

\bibitem{triebel12} H. Triebel, ``Entropy and approximation numbers of limiting embeddings, an approach via Hardy
inequalities and quadratic forms'', {\it J. Approx. Theory}, {\bf
164}:1 (2012), 31--46.

\bibitem{vas_bes} A.A. Vasil'eva, ``Kolmogorov and linear
widths of the weighted Besov classes with singularity at the
origin'', {\it J. Appr. Theory}, {\bf 167} (2013), 1--41.

\bibitem{vas_vl_raspr} A.A. Vasil'eva, ``Embedding theorem for weighted Sobolev
classes on a John domain with weights that are functions of the
distance to some $h$-set'', {\it Russ. J. Math. Phys.}, {\bf 20}:3
(2013), 360--373.

\bibitem{vas_vl_raspr2} A.A. Vasil'eva, ``Embedding theorem for weighted Sobolev
classes on a John domain with weights that are functions of the
distance to some $h$-set'', {\it Russ. J. Math. Phys.} {\bf 21}:1
(2014), 112--122.

\bibitem{vas_width_raspr} A.A. Vasil'eva, ``Widths of function classes on sets with tree-like
structure'', {\it J. Appr. Theory}, {\bf 192} (2015), 19--59.

\bibitem{vas_john} A.A. Vasil'eva, ``Widths of weighted Sobolev classes on a John domain'',
{\it Proc. Steklov Inst. Math.}, {\bf 280} (2013), 91--119.

\end{Biblio}
\end{document}